\documentclass[reqno]{amsart}

\usepackage{amssymb}
\usepackage{graphicx}
\usepackage{amscd}
\usepackage[pagebackref]{hyperref}
\usepackage{color}
\usepackage{tabularx}
\usepackage[table]{xcolor}
\usepackage{float}
\usepackage{graphics,amsmath,amssymb}
\usepackage{amsthm}
\usepackage{amsfonts}
\usepackage{latexsym}
\usepackage{epsf}
\usepackage{xifthen}
\usepackage{mathrsfs}
\usepackage{dsfont}
\usepackage{makecell}
\usepackage{subfig}
\usepackage{amsmath}
\usepackage{listings}
\usepackage{etoolbox}
\usepackage{fancyhdr}
\usepackage{pdflscape}
\usepackage[title,toc,titletoc]{appendix}
\usepackage{enumitem}
\usepackage[noadjust]{cite}
\usepackage{tikz}
\usetikzlibrary{automata,positioning,arrows}
\usepackage{young}
\usepackage[object=vectorian]{pgfornament} 
\usepackage{lipsum,tikz}
\usepackage{multirow}
\usepackage[OT2,T1]{fontenc}
\usepackage{mathtools}
\usepackage{ytableau}

\hypersetup{
	colorlinks=true, 
	linktoc=all, 
	linkcolor=blue} 

\numberwithin{equation}{section}

\theoremstyle{theorem}
\newtheorem{theorem}{Theorem}[section]
\newtheorem*{theorem*}{Theorem}

\newtheorem{corollary}[theorem]{Corollary}
\newtheorem{lemma}[theorem]{Lemma}

\newtheorem{innercustomgeneric}{\customgenericname}
\providecommand{\customgenericname}{}
\newcommand{\newcustomtheorem}[2]{%
	\newenvironment{#1}[1]
	{%
		\renewcommand\customgenericname{#2}%
		\renewcommand\theinnercustomgeneric{##1}%
		\innercustomgeneric
	}
	{\endinnercustomgeneric}
}
\newcustomtheorem{ctheorem}{Theorem}
\newcustomtheorem{clemma}{Lemma}

\theoremstyle{definition}

\newtheorem*{example*}{Example}
\newtheorem*{examples*}{Examples}

\newtheorem*{remark*}{Remark}
\newtheorem*{remarks*}{Remarks}
\newtheorem*{note*}{Note}

\newtheoremstyle{named}{}{}{\itshape}{}{\bfseries}{.}{.5em}{\thmnote{#3} #1}
\theoremstyle{named}

\newtheoremstyle{customized}{}{}{\itshape}{}{\bfseries}{.}{.5em}{\thmnote{#3}}
\theoremstyle{customized}

\newcommand{\arxiv}[1]{\href{http://arxiv.org/abs/#1}{arXiv:#1}}

\DeclareMathAlphabet{\mydutchcal}{U}{dutchcal}{m}{n}
\DeclareMathAlphabet{\myrsfso}{U}{rsfso}{m}{n}
\newcommand{\dcH}{\mydutchcal{H}}

\newcommand{\qbinom}[2]{{\genfrac{[}{]}{0pt}{}{#1}{#2}}}
\newcommand{\tqbinom}[2]{{\textstyle\genfrac{[}{]}{0pt}{}{#1}{#2}}}

\newcommand{\qangle}[2]{{\genfrac{\langle}{\rangle}{0pt}{}{#1}{#2}}}
\newcommand{\tqangle}[2]{{\textstyle\genfrac{\langle}{\rangle}{0pt}{}{#1}{#2}}}

\newcommand{\bigangle}[1]{\big\langle #1 \big\rangle}

\newcommand{\len}{\mathsf{len}}

\newcommand{\trans}{\mathsf{T}}

\newcommand{\HS}{\,\,\mathrm{HS}}

\newcommand{\Sp}{\operatorname{sp}}
\newcommand{\Spomega}{\operatorname{sp}^{\mathsf{T}}}

\title{Strong $q$-log-convexity of the $d$-Hoggatt polynomials}

\author[S. Chern]{Shane Chern}
\address[S. Chern]{Fakult\"at f\"ur Mathematik, Universit\"at Wien, Oskar-Morgenstern-Platz 1, Wien 1090, Austria}
\email{chenxiaohang92@gmail.com, xiaohangc92@univie.ac.at}

\author[W. Shi]{Wenle Shi}
\address[W. Shi]{School of Mathematical Sciences, Dalian University of Technology, Dalian 116024, P.R. China}
\email{shi-wenle@hotmail.com}

\date{}



\begin{document}
	
\sloppy

\begin{abstract}
	In this paper, three ($q$-)log-convexity or concavity properties related to the $d$-Hoggatt numbers are investigated. First, we show that the $d$-Hoggatt transformation preserves the log-convexity of a sequence. We then establish the strong $q$-log-concavity of a $q$-analog of $d$-Hoggatt numbers. Finally, we prove the core result of this work, namely, the sequence of $d$-Hoggatt polynomials is strongly $q$-log-convex, based on the theory of Schur functions.
\end{abstract}

\maketitle

\section{Introduction}

Given a sequence $(a_n)_{n\ge 0}$ of nonnegative real numbers, we say it is \emph{log-concave} if the inequality
\begin{align*}
	a_n^2\ge a_{n-1}a_{n+1}
\end{align*}
holds for all $n\ge 1$, and \emph{log-convex} if
\begin{align*}
	a_{n-1}a_{n+1}\ge a_n^2.
\end{align*}
If we lift elements in the sequence to polynomials in $\mathbb{R}[q]$ with $q$ an indeterminate and define the $q$-analog of the order relation ``$\ge$'' by $f(q)\ge_q g(q)$ provided that $f(q)-g(q)\in \mathbb{R}_{\ge 0}[q]$, then the concept of log-concavity and log-convexity is naturally extended to \emph{$q$-log-concavity} and \emph{$q$-log-convexity}. Sagan~\cite[p.~290]{Sag1992} further introduced the nontion of strong $q$-log-concavity: a sequence $(a_n(q))_{n\ge 0}$ of polynomials over $\mathbb{R}$ is \emph{strongly $q$-log-concave} if
\begin{align*}
	a_n(q)a_m(q)\ge_q a_{n+1}(q)a_{m-1}(q)
\end{align*}
for all $n\ge m\ge 1$. In the same vein, we say the sequence $(a_n(q))_{n\ge 0}$ is \emph{strongly $q$-log-convex} if
\begin{align*}
	a_{n+1}(q)a_{m-1}(q)\ge_q a_n(q)a_m(q).
\end{align*}
In particular, for ordinary $q$-log-concavity and $q$-log-convexity, we only require the two inequalities for the case where $m$ and $n$ are equal.

Inspired by a conjecture of Verner Hoggatt in private communication, Fielder and Alford~\cite[Section~4]{FA1989} introduced the \emph{$d$-Hoggatt triangle} $\big(\tqangle{n}{k}^{(d)}\big)_{0\le k\le n}$, which specializes to the \emph{Pascal triangle} $\big(\binom{n}{k}\big)_{0\le k\le n}$ at $d=1$. The original definition of Fielder and Alford~\cite[p.~163, eq.~(13)]{FA1989} is recursive. However, it is easy to derive the following product formula in terms of binomial coefficients from their recurrence relation:
\begin{align}\label{eq:d-H-def}
	\qangle{n}{k} = \qangle{n}{k}^{(d)} := \prod_{j=0}^{n-1} \frac{\binom{n+j}{k}}{\binom{k+j}{k}},
\end{align}
valid for $n,k\ge 0$. It is known that the $2$- and $3$-Hoggatt triangles are the \emph{Narayana triangle}~\cite[Entry~A001263]{OEIS} and the \emph{Baxter triangle}~\cite[Entry~A056939]{OEIS}, respectively.

For a triangular array of combinatorial numbers $(S_{n,k})_{0\leq k \leq n}$, we define a linear transformation that maps a sequence $(a_n)_{n\ge0}$ to a new sequence $(c_n)_{n\ge0}$ by 
\begin{align*}
	c_n=\sum_{k=0}^{n}S_{n,k}a_k.
\end{align*}
We know from the \emph{Davenport--P\'{o}lya theorem}~\cite[p.~1, Theorem~2]{DP1949} that if the sequences $(a_n)_{n\geq0}$ and $(b_n)_{n\geq0}$ are log-convex, then so is their binomial convolution 
\begin{align*}
	c_n=\sum_{k=0}^{n}\binom{n}{k}a_kb_{n-k}.
\end{align*}
Motivated by this phenomenon, Liu and Wang~\cite[p.~474, Conjecture~5.2]{LW2007} conjectured that the Narayana transformation also preserves the log-convexity. This prediction was later confirmed by Chen, Wang, and Yang~\cite[p.~318, Theorem~4.3]{CWY2010}. Quite recently, Liu and Mao~\cite[Theorem~6.1]{LM2026} further discovered the same behavior of the Baxter transformation. In this work, we provide a unified treatment on the generic $d$-Hoggatt transformation.

\begin{theorem}\label{thm:main1}
	Let $d\geq 1$ be fixed. If the sequence $(a_n)_{n\geq0}$ of nonnegative real numbers is log-convex, then so is the sequence $(c_n)_{n\geq0}$ where
	\begin{align*}
		c_n=\sum_{k=0}^{n}\qangle{n}{k}^{(d)}a_k.
	\end{align*}
\end{theorem}

Note that in the proof of the Narayana case, the basic idea of Chen, Wang, and Yang builds on the strong $q$-log-convexity of the Narayana polynomials~\cite[p.~311, Theorem~3.2]{CWY2010}, the weaker version of which was first conjectured by Liu and Wang~\cite[p.~474, Conjecture~5.1]{LW2007}. Later on, Zhu~\cite[p.~605, Proposition~3.20]{Zhu2013} offered an alternative proof of this strong $q$-log-convexity. For Baxter polynomials, we recently witnessed their strong $q$-log-convexity in the work of Liu and Mao~\cite[Theorem~5.1]{LM2026}. Now let us define the \emph{$d$-Hoggatt polynomials} by
\begin{align*}
	H_n(q) = H^{(d)}_n(q) := \sum_{k=0}^{n}\qangle{n}{k}^{(d)}q^k.
\end{align*}
Mao and the second author \cite[p.~17]{MS2026} conjectured that $\big(H^{(d)}_n(q)\big)_{n\geq0}$ is $q$-log-convex for fixed $d\ge 1$. Here we not only provide an affirmative answer to this problem but also establish a strengthening toward strong $q$-log-convexity, which will serve as an essential tool in our proof of Theorem~\ref{thm:main1}. This is also the most central result of the present work.

\begin{theorem}\label{thm:main2}
	Let $d\geq 1$ be fixed. The sequence $\big(H^{(d)}_n(q)\big)_{n\geq0}$ is strongly $q$-log-convex.
\end{theorem}

Cigler~\cite{Cig2021} recently considered a $q$-analog of the $d$-Hoggatt numbers:\footnote{We multiply the original definition of Cigler~\cite[p.~10, eq.~(22)]{Cig2021} by the prefactor $q^{d\binom{k}{2}}$. This is consistent with Cigler's remark on \cite[p.~14]{Cig2021} and the $q$-Narayana numbers studied by Chen, Wang, and Yang~\cite[p.~306]{CWY2010}.}
\begin{align}\label{eq:q-H-def}
	\qangle{n}{k}_q = \qangle{n}{k}^{(d)}_q := q^{d\binom{k}{2}}\prod_{j=0}^{d-1}\frac{\qbinom{n+j}{k}_q}{\qbinom{k+j}{k}_q},
\end{align}
where the \emph{$q$-binomial coefficients} are given by
\begin{align*}
	\qbinom{n}{k}_q:=\frac{[n]_q!}{[k]_q![n-k]_q!},
\end{align*}
with
\begin{align*}
	[m]_q:=\frac{1-q^m}{1-q}, \qquad\qquad [m]_q!:=[1]_q[2]_q\cdots[m]_q.
\end{align*}

For the sequence $\big(\tqbinom{n}{k}_q\big)_{0\le k\le n}$ of $q$-binomial coefficients with $n$ fixed, its strong $q$-log-concavity was shown by Butler~\cite[p.~59, Theorem~4.2]{But1990} and Sagan~\cite[p.~292, Theorem~2.3]{Sag1992}. Simultaneously, Butler~\cite[p.~62, Section~5]{But1990} proved that the sequence $\big(\tqbinom{n}{k}_q\big)_{n\ge k}$ with $k$ fixed is strongly $q$-log-concave via a partition-theoretic approach and Sagan~\cite[p.~297, Corollary~3.4(2)]{Sag1992} provided another proof by the theory of symmetric functions. Around the same time, Krattenthaler~\cite[p.~334, Theorem~1]{Kra1989}, independently, established an even stronger result that for $n\ge m\ge 0$ and $k\ge l \ge 0$,
\begin{align*}
	\qbinom{n}{l}_q \qbinom{m}{k}_q\ge_q \qbinom{n}{l-1}_q \qbinom{m}{k+1}_q.
\end{align*}
A parallel inequality
\begin{align*}
	\qbinom{n}{l}_q \qbinom{m}{k}_q\ge_q \qbinom{n+1}{l}_q \qbinom{m-1}{k}_q,
\end{align*}
valid for nonnegative integers $n,m,l,k$ with $k\ge l$ and $n+k\ge m+l$, was discovered by the first author~\cite[p.~34, Theorem~1.3]{Che2023}. Moving forward to the $q$-Narayana numbers, a similar strong $q$-log-concavity property was discovered by Chen, Wang, and Yang~\cite[p.~320, Theorem~5.2 and p.~323, Theorem~5.6]{CWY2010} for fixed $n$ or fixed $k$. Recently, Liu and Mao~\cite[Theorems~6.6 and 6.7]{LM2026} further proved that for $d\ge 1$, the $q$-analog $\tqangle{n}{k}^{(d)}_q$ of the $d$-Hoggatt numbers is $q$-log-concave for fixed $n$ or fixed $k$. In this work, we strengthen this discovery to strong $q$-log-concavity.

\begin{theorem}\label{thm:main3}
	Let $d\geq 1$ and $n\ge 0$ be fixed. The sequence $\big( \tqangle{n}{k}^{(d)}_q\big)_{0\le k\le n}$ is strongly $q$-log-concave.
\end{theorem}

\begin{theorem}\label{thm:main4}
	Let $d\geq 1$ and $k\ge 0$ be fixed. The sequence $\big( \tqangle{n}{k}^{(d)}_q\big)_{n\ge k}$ is strongly $q$-log-concave.
\end{theorem}

\textbf{Outline of the paper.} In Section~\ref{sec:trans}, we study the $d$-Hoggatt transformation and prove Theorem~\ref{thm:main1} under the assumption of Theorem~\ref{thm:main2}. We then collect some preliminary results in Section~\ref{sec:Schur} on the theory of symmetric functions, especially of Schur functions. Building on these results, we demonstrate Theorems~\ref{thm:main3} and \ref{thm:main4} in Section~\ref{sec:q-Hoggatt}. Finally, we move on to the most difficult task in this work, namely, showing the strong $q$-log-convexity of the $d$-Hoggatt polynomials. To begin with, we establish the Schur positivity of an auxiliary family of functions in Section~\ref{sec:Schur-positivity}. Then Theorem~\ref{thm:main2} will be established in Section~\ref{sec:d-Hoggatt}.

In subsequent sections, we always fix $d\ge 1$ and omit the superscript ``$(d)$'' if the value of $d$ does not interfere with our analysis.

\section{$d$-Hoggatt transformation}\label{sec:trans}

In this section, we show that the $d$-Hoggatt transformation preserves the log-convexity. To begin with, we require the \emph{Pochhammer symbols} for $m\in \mathbb{N}\cup\{\infty\}$:
\begin{align*}
	(x)_m := \prod_{k=0}^{m-1} (x+k).
\end{align*}
It is clear from \eqref{eq:d-H-def} that
\begin{align}\label{eq:d-H-ratio}
	\frac{\qangle{n+1}{k}}{\qangle{n}{k}} = \frac{(n+1)_d}{(n-k+1)_d}, \qquad\qquad \frac{\qangle{n-1}{k}}{\qangle{n}{k}} = \frac{(n-k)_d}{(n)_d}.
\end{align}

For $n\geq1$, $0\leq r\leq2n$, and
$0\leq k\leq\lfloor \frac{r}{2}\rfloor$, let us define
\begin{align}
	\alpha_k(n,r) := \begin{cases}
		\qangle{n+1}{k}\qangle{n-1}{r-k} + \qangle{n+1}{r-k}\qangle{n-1}{k} - 2\qangle{n}{k}\qangle{n}{r-k}, & \text{if $k<\frac{r}{2}$},\\[6pt]
		\qangle{n+1}{k}\qangle{n-1}{k} - \qangle{n}{k}^2, & \text{if $k=\frac{r}{2}$}.
	\end{cases}
\end{align}
It is simple to verify that
\begin{align*}
	H_{n-1}(q)H_{n+1}(q) - H_n(q)^2 = \sum_{r=0}^{2n}\sum_{k=0}^{\lfloor \frac{r}{2}\rfloor} \alpha_k(n,r)q^r.
\end{align*}
In view of the $q$-log-convexity of $(H_n(q))_{n\geq0}$ stated in Theorem~\ref{thm:main2},
\begin{align*}
	H_{n-1}(q)H_{n+1}(q) - H_n(q)^2 \in \mathbb{N}[q],
\end{align*}
we know that
\begin{align}\label{eq:alpha-sum>=0}
	\sum_{k=0}^{\lfloor \frac{r}{2}\rfloor} \alpha_k(n,r) \ge 0.
\end{align}
Therefore, if we fix $n\geq1$ and $0\leq r\leq2n$, then it is impossible that $\alpha_k(n,r) < 0$ for every $k$ with $0\leq k\leq\lfloor \frac{r}{2}\rfloor$. In what follows, we characterize the sign change of $\alpha_k(n,r)$.

\begin{lemma}\label{le:k'}
	For given $n\geq1$ and $0\leq r\leq2n$, there exists an integer $k' := k'(n,r)$ such that
	\begin{align*}
		\alpha_k(n,r) \begin{cases}
			\ge 0, & \text{when $0\le k\le k'$},\\
			\le 0, & \text{when $k'< k\le \lfloor \frac{r}{2}\rfloor$}.
		\end{cases}
	\end{align*}
\end{lemma}

\begin{proof}
	Throughout, we always assume $k\ge 0$. We start by noticing that
	\begin{align*}
		\alpha_k(n,r) \begin{cases}
			= 0, & \text{when $k\le r-n-2$},\\
			> 0, & \text{when $k= r-n-1$}.
		\end{cases}
	\end{align*}
	For the first situation, we use the fact that $\tqangle{n-1}{r-k} = \tqangle{n}{r-k} = \tqangle{n+1}{r-k} = 0$ when $k\le r-n-2$. For the second situation, we note that $k\ne \frac{r}{2}$ and $k\le n-1$ because $r\le 2n$. Thus, $\alpha_{k}(n,r) = \qangle{n-1}{k} > 0$. In the sequel, suppose $k\ge r-n$. Let us introduce an auxiliary sequence indexed by $k$:
	\begin{align*}
		\beta_k(n,r) := \frac{(n-k)_d}{(n-r+k+1)_d} + \frac{(n-r+k)_d}{(n-k+1)_d}.
	\end{align*}
	In light of \eqref{eq:d-H-ratio},
	\begin{align*}
		\frac{n+d}{n} \beta_k(n,r) - 2 = \begin{cases}
			\frac{1}{\qangle{n}{k}\qangle{n}{r-k}}\alpha_k(n,r), & \text{if $k\le \frac{r}{2}$},\\[10pt]
			\frac{2}{\qangle{n}{k}^2}\alpha_k(n,r), & \text{if $k= \frac{r}{2}$}.
		\end{cases}
	\end{align*}
	For consecutive indices $k,k+1$ satisfying $k\ge r-n$ and $k+1\leq\lfloor \frac{r}{2}\rfloor$, we have
	\begin{align*}
		n-k> n-r+k+1\ge 1,
	\end{align*}
	so that
	\begin{align*}
		(n-k)_{d-1}(n-k)_{d+1} \geq (n-r+k+1)_{d-1}(n-r+k+1)_{d+1},
	\end{align*}
	which further gives us the nonnegativity of
	\begin{align*}
		\beta_k(n,r) - \beta_{k+1}(n,r) = d(2n-r+d)
		\left(\frac{(n-k)_{d-1}}{(n-r+k+1)_{d+1}}
		-\frac{(n-r+k+1)_{d-1}}{(n-k)_{d+1}}\right).
	\end{align*}
	Thus, the numbers $\beta_k(n,r)$ are nonincreasing in $k$, and so are $\frac{n+d}{n} \beta_k(n,r) - 2$. We conclude that when $k$ runs from $0$ to $\lfloor \frac{r}{2}\rfloor$, the sign of $\alpha_k(n,r)$ changes at most once, from nonnegative to nonpositive. The claimed characterization then follows.
\end{proof}

Now we are in a position to prove Theorem~\ref{thm:main1}.

\begin{proof}[Proof of Theorem~\ref{thm:main1}]
	Note that
	\begin{align*}
		c_{n-1}c_{n+1} - c_n^2 = \sum_{r=0}^{2n}\sum_{k=0}^{\lfloor \frac{r}{2}\rfloor} \alpha_k(n,r) a_ka_{r-k}.
	\end{align*}
	Since $(a_k)_{k\geq0}$ is a log-convex sequence of nonnegative real numbers, we obtain that
	\begin{align*}
		a_0a_r\geq a_1a_{r-1}\geq \cdots \geq a_{\lfloor \frac{r}{2}\rfloor}a_{r-\lfloor \frac{r}{2}\rfloor}\ge 0.
	\end{align*}
	Now choosing $k'$ as in Lemma~\ref{le:k'}, we see that for every $k$ with $0\leq k\leq\lfloor \frac{r}{2}\rfloor$,
	\begin{align*}
		\alpha_k(n,r) (a_ka_{r-k} - a_{k'}a_{r-k'})\ge 0.
	\end{align*}
	Therefore,
	\begin{align*}
		\sum_{k=0}^{\lfloor \frac{r}{2}\rfloor} \alpha_k(n,r) a_ka_{r-k} = a_{k'}a_{r-k'}\sum_{k=0}^{\lfloor \frac{r}{2}\rfloor} \alpha_k(n,r) + \sum_{k=0}^{\lfloor \frac{r}{2}\rfloor} \alpha_k(n,r) (a_ka_{r-k} - a_{k'}a_{r-k'}) \ge 0,
	\end{align*}
	where we have recalled \eqref{eq:alpha-sum>=0} for the nonnegativity of the first sum. Summing over $r$ gives the desired log-convexity of $c_n$.
\end{proof}

\section{Preliminaries on Schur functions}\label{sec:Schur}

In this section, we give a brief overview of the theory of Schur functions, most of which can be found in the monograph of Stanley~\cite[Chapter~7]{Sta1999}.

We begin with the combinatorial setting. A \emph{partition} is a weakly decreasing sequence $\lambda=(\lambda_1,\lambda_2,\ldots,\lambda_l)$ of positive integers. If the sequence is empty, we say this partition is the \emph{empty partition} $\varnothing$. We write $|\lambda|$ for its \emph{size} $\sum_{i=1}^l \lambda_i$. Also, the \emph{length} of $\lambda$, denoted by $\len(\lambda)$, is defined as $l$, and the \emph{width} is defined as $\lambda_1$. When an index $i$ exceeds the length of $\lambda$, the corresponding part $\lambda_i$ is understood to be $0$. The \emph{Young diagram} of $\lambda$ is obtained by depicting left-aligned boxes in rows such that there are $\lambda_i$ boxes in the $i$-th row. By flipping the Young diagram along the main diagonal, we obtain the Young diagram of a new partition, which is called the \emph{conjugate} of $\lambda$, denoted by $\lambda^\trans$. For two partitions $\mu$ and $\lambda$, the notation $\mu \subseteq \lambda$ means $\mu_i\le \lambda_i$ for every index $i$. The \emph{skew diagram} $\lambda/\mu$ is obtained by removing the Young diagram of $\mu$ from the upper-left corner of the diagram of $\lambda$. We also write $(\lambda/\mu)^\trans := \lambda^\trans/\mu^\trans$.

Assuming $\mu \subseteq \lambda$, a \emph{semistandard Young tableau} of shape $\lambda/\mu$ is a filling $T$ of the cells in the skew diagram of $\lambda/\mu$ by positive integers that is weakly increasing from left to right in every row and strictly increasing from top to bottom in every column. Its \emph{type} is the sequence $(m_1,m_2,\ldots)$, where $m_i$ denotes the \emph{multiplicity} of $i$ in the filling $T$. Let $X:=(x_1,x_2,\ldots)$ be a list of variables, known as an \emph{alphabet}, and put
\begin{align*}
	X^T := \prod_{i\ge 1} x_i^{m_i}.
\end{align*}
The combinatorial definition of the \emph{skew Schur function} indexed by $\lambda/\mu$ is \cite[p.~310, Definition~7.10.1]{Sta1999}:
\begin{align*}
	s_{\lambda/\mu} = s_{\lambda/\mu}(X) := \sum_{T} X^T,
\end{align*}
summing over all semistandard Young tableaux of shape $\lambda/\mu$. For $\mu=\varnothing$, we have the ordinary \emph{Schur function} $s_\lambda = s_\lambda(X)$. We also use the convention $s_{\lambda/\mu}=0$ if $\mu\not\subseteq \lambda$.

Schur functions play a central role in the theory of symmetric functions. Let $\Lambda$ be the ring of symmetric functions over $\mathbb{Z}$. An important fact is that the Schur functions $s_\lambda$ form a \emph{basis} for $\Lambda$; see \cite[p.~315, Corollary~7.10.6]{Sta1999}. Moreover, this basis is \emph{orthogonal}, given by the \emph{Hall inner product}~\cite[p.~336]{Sta1999}:
\begin{align*}
	\bigangle{s_\lambda, s_\mu} = \begin{cases}
		1, & \text{if $\lambda =\mu$},\\
		0, & \text{if $\lambda \ne\mu$}.
	\end{cases}
\end{align*}
Let $F = \sum_{\lambda} a_\lambda s_\lambda \in \Lambda$ be arbitrary, expanded by the Schur basis. We write
\begin{align*}
	[s_\lambda] F := a_\lambda = \bigangle{s_\lambda, F}.
\end{align*}
If $[s_\lambda] F \ge 0$ for every partition $\lambda$, then we say $F$ is \emph{Schur positive}, expressed by $F\ge_s 0$.

We know from \cite[p.~338, eq.~(7.64)]{Sta1999} that
\begin{align}
	s_\mu s_\nu &= \sum_{\lambda} c_{\mu,\nu}^{\lambda} s_\lambda,\label{eq:Schur-c-1}\\
	s_{\lambda/\nu} &= \sum_{\mu} c_{\mu,\nu}^{\lambda} s_\mu,\label{eq:Schur-c-2}\\
	s_{\lambda/\mu} &= \sum_{\nu} c_{\mu,\nu}^{\lambda} s_\nu,\label{eq:Schur-c-3}
\end{align}
using the same \emph{Littlewood--Richardson coefficients} $c_{\mu,\nu}^{\lambda}$. In particular, for arbitrary $F\in \Lambda$,
\begin{align}\label{eq:Hall-F}
	\bigangle{s_{\lambda/\mu}, F} = \bigangle{s_\lambda, s_\mu F},
\end{align}
both equal to $\sum_{\nu} \big(c_{\mu,\nu}^{\lambda}\cdot [s_\nu] F\big)$.

To further understand the Littlewood--Richardson coefficients, we require the Littlewood--Richardson rule. We first obtain the \emph{reverse row word} of a tableau by reading each row from right to left, beginning with the top row and then proceeding downward. Moreover, a word is a \emph{lattice word} if in every initial segment, the multiplicity of $i$ is at least the multiplicity of $i+1$ for every $i\ge 1$. Now a \emph{Littlewood--Richardson tableau} is a semistandard Young tableau whose reverse row word is a lattice word. It follows from this definition that the type of a Littlewood--Richardson tableau forms a partition.

The \emph{Littlewood--Richardson rule}~\cite[p.~432, Theorem~A1.3.3]{Sta1999} states that $c_{\mu,\nu}^{\lambda}$ equals the number of Littlewood--Richardson tableaux of shape $\lambda/\mu$ and type $\nu$, yielding that $c_{\mu,\nu}^{\lambda}\ge 0$, and hence that
\begin{enumerate}[label={\textup{(\arabic{*})}}, leftmargin=*, labelsep=5pt, align=left, itemsep=2pt, topsep=2pt]
	\item Every skew Schur function is Schur positive;
	
	\item A product of Schur positive symmetric functions is
	Schur positive.
\end{enumerate}

We collect some additional facts about the Littlewood--Richardson tableaux and coefficients.

\begin{lemma}\label{le:LR-facts}
	Every entry in the $i$-th row of a Littlewood--Richardson tableau is no larger than $i$. Also, if $c_{\mu,\nu}^{\lambda}>0$, then $\mu,\nu \subseteq \lambda$ and for every $i\ge 1$,
	\begin{align}\label{eq:LR-facts-ineq}
		\lambda_i \le \mu_1 + \nu_i.
	\end{align}
\end{lemma}

\begin{proof}
	We prove the first assertion inductively on $i$. For the first row, the rightmost cell must be filled with $1$ by the Littlewood--Richardson condition, and hence all entries in this row are $1$ by weak monotonicity. Suppose all entries in rows above the $i$-th row are at most $i-1$. Then the rightmost entry in the $i$-th row cannot exceed $i$; otherwise, the Littlewood--Richardson condition is violated at this cell. We then conclude that all entries in this row are at most $i$ by weak monotonicity.
	
	For the second assertion, since $c_{\mu,\nu}^{\lambda}>0$, we have at least one Littlewood--Richardson tableau of shape $\lambda/\mu$ and type $\nu$. Then $\mu\subseteq \lambda$ from the shape. Meanwhile, in \eqref{eq:Schur-c-1}, switching $\mu$ and $\nu$ does not change the left-hand side but turns $c_{\mu,\nu}^{\lambda}$ to $c_{\nu,\mu}^{\lambda}$ on the right-hand side. Thus,
	\begin{align*}
		c_{\mu,\nu}^{\lambda} = c_{\nu,\mu}^{\lambda}.
	\end{align*}
	Then $\nu \subseteq \lambda$ is also true. Next, if $\lambda_i\le \mu_1$, then \eqref{eq:LR-facts-ineq} is immediate. If $\lambda_i> \mu_1$, then in the $\lambda_i-\mu_1$ columns indexed from $\mu_1+1$ to $\lambda_i$, there are skew cells in all rows from the first to the $i$-th. In each of these columns, column strictness makes the entry in the $i$-th row at least $i$, while the first part of the lemma makes it at most $i$, so this entry has to be $i$. Then there are at least $\lambda_i-\mu_1$ cells filled with $i$, thereby yielding \eqref{eq:LR-facts-ineq}.
\end{proof}

Define the \emph{rectangular partition}
\begin{align*}
	(m^l) := (\underbrace{m,\ldots,m}_{\text{$l$ parts}}).
\end{align*}
The \emph{complete homogeneous symmetric functions} $h_r$ and the \emph{elementary symmetric functions} $e_r$, defined respectively by
\begin{align}\label{eq:h/e-def}
	h_r := s_{(r)}, \qquad\qquad e_r := s_{(1^r)},
\end{align}
are the most basic Schur functions indexed by rectangular partitions. Here we adopt the convention that $h_0=e_0=1$ and $h_r=e_r=0$ for $r<0$.

We also need a special shape of skew diagrams for later use. A skew diagram $\lambda/\mu$ is a \emph{horizontal strip} if it has at most one cell in each column. We write
\begin{align*}
	\text{``$\lambda/\mu \HS$''}
\end{align*}
for this condition. In particular, this condition is equivalent to
\begin{align}\label{eq:HS-condition}
	\lambda_1\ge \mu_1\ge \lambda_2\ge \mu_2 \ge \lambda_3 \ge \cdots.
\end{align}
According to the product form of the \emph{Pieri rule}~\cite[p.~339, Theorem~7.15.7]{Sta1999},
\begin{align}\label{eq:Pieri-product}
	h_r s_\mu = \sum_{\substack{\lambda\colon \lambda/\mu\HS\\|\lambda|-|\mu|=r}} s_\lambda.
\end{align}
Also, the skew form of the Pieri rule~\cite[p.~340, Corollary~7.15.9]{Sta1999} states that
\begin{align}\label{eq:Pieri-skew}
	s_{\lambda/(r)} = \sum_{\substack{\mu\colon \lambda/\mu\HS\\|\lambda|-|\mu|=r}} s_\mu.
\end{align}

Now we move on to the Schur positivity related to rectangular partitions.

\begin{lemma}\label{le:Schur-rectangular-1}
	Let $u\ge 1$ and $a\ge b\ge 1$. Then
	\begin{align}\label{eq:Schur-rectangular-1}
		s_{(a^u)}s_{(b^u)} - s_{((a+1)^u)}s_{((b-1)^u)} \ge_s 0.
	\end{align}
\end{lemma}

\begin{proof}
	By \eqref{eq:Schur-c-1},
	\begin{align*}
		s_{(a^u)}s_{(b^u)} - s_{((a+1)^u)}s_{((b-1)^u)} = \sum_{\lambda} \big(c_{(a^u),(b^u)}^{\lambda}-c_{((a+1)^u),((b-1)^u)}^{\lambda}\big)s_\lambda.
	\end{align*}
	It is sufficient to show that for every partition $\lambda$,
	\begin{align}\label{eq:c-diff}
		c_{(a^u),(b^u)}^{\lambda}-c_{((a+1)^u),((b-1)^u)}^{\lambda} \ge 0.
	\end{align}
	If $((a+1)^u)\not\subseteq \lambda$, then $c_{((a+1)^u),((b-1)^u)}^{\lambda}=0$ according to the Littlewood--Richardson rule, and hence \eqref{eq:c-diff} holds. Now assume that $((a+1)^u)\subseteq \lambda$ so that $\lambda_i\ge a+1$ for every index $i$ with $1\le i\le u$. If $c_{((a+1)^u),((b-1)^u)}^{\lambda}=0$, then we still have nothing to prove. Hence, we further assume that $c_{((a+1)^u),((b-1)^u)}^{\lambda}>0$. We claim that $\lambda_{u+1}< a+1$. Note that the Littlewood--Richardson rule ensures the existence of a Littlewood--Richardson tableau of shape $\lambda/((a+1)^u)$ and type $((b-1)^u)$. By column strictness and the first assertion in Lemma~\ref{le:LR-facts}, we know that all cells in the $i$-th row of $\lambda/((a+1)^u)$ are filled with $i$. Also, \eqref{eq:LR-facts-ineq} in Lemma~\ref{le:LR-facts} tells us that $\lambda_i\le (a+1)+(b-1)=a+b$. Next, by the Littlewood--Richardson condition, we may fill the $(u+1)$-th row of $\lambda/((a+1)^u)$ with $u$ of multiplicity at most $\lambda_{u-1} - \lambda_{u}$, preceded by $u-1$ of multiplicity at most $\lambda_{u-2} - \lambda_{u-1}$, and so forth, till $1$ of multiplicity $\lambda_0-\lambda_1$ where $\lambda_0 = a+b$. Thus,
	\begin{align*}
		\lambda_{u+1}\le \lambda_0 - \lambda_u\le (a+b)-(a+1) = b-1 < a+1
	\end{align*}
	since $b\le a$. Now for any Littlewood--Richardson tableau of shape $\lambda/((a+1)^u)$ and type $((b-1)^u)$, we can injectively obtain a Littlewood--Richardson tableau of shape $\lambda/(a^u)$ and type $(b^u)$ by preserving the fillings of the former tableau and then filling the $(a+1)$-th column of $\lambda/(a^u)$, which contains exactly $u$ cells, with $1,2,\ldots,u$, from top to bottom. Therefore, \eqref{eq:c-diff} is also valid in this situation.
\end{proof}

Let $\omega$ be the \emph{standard involution} on $\Lambda$, defined by
\begin{align*}
	\omega(s_\lambda) := s_{\lambda^\trans},
\end{align*}
which preserves the Schur positivity. Since $\omega$ is a ring automorphism, it is true that for any $F,G\in \Lambda$,
\begin{align}\label{eq:omega-product}
	\omega(FG) = \omega(F)\omega(G).
\end{align}
Also, we know from \cite[p.~338, Theorem~7.15.6]{Sta1999} that
\begin{align}\label{eq:omega-skew}
	\omega(s_{\lambda/\mu}) = s_{\lambda^\trans/\mu^\trans}.
\end{align}

We then arrive at a companion to Lemma~\ref{le:Schur-rectangular-1}, which also appeared in \cite[Remark~7.2]{BM2004}.

\begin{lemma}\label{le:Schur-rectangular-2}
	Let $u\ge 1$ and $a\ge b\ge 1$. Then
	\begin{align}\label{eq:Schur-rectangular-2}
		s_{(u^a)}s_{(u^b)} - s_{(u^{a+1})}s_{(u^{b-1})} \ge_s 0.
	\end{align}
\end{lemma}

\begin{proof}
	Applying $\omega$ to the left-hand side of \eqref{eq:Schur-rectangular-1} produces the left-hand side of \eqref{eq:Schur-rectangular-2}, where we have used \eqref{eq:omega-product}. Since $\omega$ preserves the Schur positivity, the desired statement follows from Lemma~\ref{le:Schur-rectangular-1}.
\end{proof}

The preceding content works for generic alphabets $X=(x_1,x_2,\ldots)$. Now let $X_N := (x_1,\ldots,x_N)$ be finite. The \emph{specialization map} $\Sp_N$ sets all variables after $x_N$ to $0$, so that
\begin{align*}
	\Sp_N(F) := F(X_N).
\end{align*}
We introduce a companion operator that will be frequently utilized:
\begin{align*}
	\Spomega_N(F) := \Sp_N\big(\omega(F)\big).
\end{align*}
In particular, \eqref{eq:omega-product} tells us that for any $F,G\in \Lambda$,
\begin{align}\label{eq:spt-product}
	\Spomega_N(FG) = \Spomega_N(F)\Spomega_N(G).
\end{align}
Finally, we define the \emph{projection} onto Schur functions of width at most $N$ by
\begin{align*}
	\pi_N\left(\sum_{\lambda} a_\lambda s_\lambda\right) := \sum_{\lambda\colon \lambda_1\le N} a_\lambda s_\lambda.
\end{align*}

Let
\begin{align*}
	I_N := \operatorname{span}\{s_\lambda\colon \lambda_1>N\}.
\end{align*}

\begin{lemma}\label{le:I-ideal}
	The span $I_N$ is an ideal of $\Lambda$.
\end{lemma}

\begin{proof}
	Suppose $s_\lambda$ occurs in the Schur expansion of $s_\mu s_\nu$ so that $c_{\mu,\nu}^{\lambda} > 0$. According to Lemma~\ref{le:LR-facts}, $\mu,\nu\subseteq \lambda$. Thus, if $\mu_1>1$ or $\nu_1>1$, we must have $\lambda_1>1$. This means that multiplication of an element of $I_n$ and an arbitrary symmetric function in $\Lambda$ remains in $I_N$, thereby confirming that $I_N$ is an ideal.
\end{proof}

\begin{lemma}\label{le:anni}
	Both $\Spomega_N$ and $\pi_N$ annihilate all functions in $I_N$.
\end{lemma}

\begin{proof}
	The annihilation of $\pi_N$ follows by definition. Next, $\Spomega_N(s_\lambda) = s_{\lambda^\trans}(X_N)$, which vanishes when $\len(\lambda^\trans)>N$ since in this case the first column already requires more than $N$ distinct variables. In this case, we have $\lambda_1 > N$, yielding the annihilation of $\Spomega_N$ on $I_N$.
\end{proof}

\begin{lemma}\label{le:spt-proj}
	For any $F\in \Lambda$,
	\begin{align}\label{eq:spt-proj}
		\Spomega_N(F) = \Spomega_N\big(\pi_N(F)\big).
	\end{align}
	Also, $\Spomega_N$ is injective on the span of $s_\lambda$ with $\lambda_1\le N$.
\end{lemma}

\begin{proof}
	Note that
	\begin{align*}
		F = \pi_N(F) + \big(F-\pi_N(F)\big),
	\end{align*}
	with $F-\pi_N(F) \in I_N$. We then use the annihilation of $\Spomega_N$ on $I_N$ to obtain the claimed relation \eqref{eq:spt-proj}. For the second assertion, we only need the fact that the Schur polynomials $s_{\lambda^\trans}(X_N)$ with $\lambda_1\le N$ form a basis of symmetric polynomials in $N$ variables.
\end{proof}

\begin{lemma}
	For any $F,G\in \Lambda$,
	\begin{align}\label{eq:FG-proj}
		\pi_N(FG) = \pi_N\big(F\pi_N(G)\big).
	\end{align}
\end{lemma}

\begin{proof}
	Note that
	\begin{align*}
		FG = F\cdot \pi_N(G) + F\cdot \big(G-\pi_N(G)\big).
	\end{align*}
	Here $G-\pi_N(G)$ is in $I_N$, and so is $F\cdot \big(G-\pi_N(G)\big)$ since $I_N$ is an ideal of $\Lambda$. Applying $\pi_N$ to both sides and using the annihilation of $\pi_N$ on $I_N$, the claimed relation follows.
\end{proof}

Finally, we show how Schur functions are connected with $d$-Hoggatt numbers and their $q$-analogs. To begin with, we specify two alphabets:
\begin{align*}
	Q_N := (1,q,\ldots,q^{N-1}), \qquad\qquad 1_N := (\underbrace{1,1,\ldots,1}_{\text{$N$ copies}}).
\end{align*}

\begin{lemma}
	We have
	\begin{align}\label{eq:Hoggatt-Schur-q}
		s_{(d^k)}(Q_n) = \qangle{n}{k}_q^{(d)}.
	\end{align}
	In particular,
	\begin{align}\label{eq:Hoggatt-Schur-1}
		s_{(d^k)}(1_n) = \qangle{n}{k}^{(d)}.
	\end{align}
\end{lemma}

\begin{proof}
	By the hook-content formula~\cite[p.~374, Theorem~7.21.2]{Sta1999},
	\begin{align*}
		s_{(d^k)}(Q_n) = q^{d\binom{k}{2}} \prod_{i=1}^k \prod_{j=1}^d \frac{[n+j-i]_q}{[d+k-i-j-1]_q} = q^{d\binom{k}{2}} \prod_{j=0}^{d-1} \prod_{i=1}^k \frac{[n-k+i+j]_q}{[i+j]_q}.
	\end{align*}
	To get the last equality, we have interchanged the products, and then made the change of indices $(i,j)\mapsto (k-i+1,j+1)$ for the numerator and $(i,j)\mapsto (k-i+1,d-j)$ for the denominator. Now
	\begin{align*}
		s_{(d^k)}(Q_n) = q^{d\binom{k}{2}} \prod_{j=0}^{d-1} \frac{\qbinom{n+j}{k}_q}{\qbinom{k+j}{k}_q},
	\end{align*}
	thereby giving \eqref{eq:Hoggatt-Schur-q} by recalling \eqref{eq:q-H-def}. For \eqref{eq:Hoggatt-Schur-1}, we only need the specialization at $q=1$.
\end{proof}

\begin{lemma}
	We have, for $r\ge 0$,
	\begin{align}\label{eq:Hoggatt-Schur-q-r}
		s_{(r^k)}(Q_{d+k}) = q^{(r-d)\binom{k}{2}}\qangle{k+r}{k}_q^{(d)}.
	\end{align}
\end{lemma}

\begin{proof}
	If $r=0$, then both sides of \eqref{eq:Hoggatt-Schur-q-r} are $1$. Now assume $r\ge 1$. We have shown in the preceding proof that
	\begin{align*}
		s_{(r^k)}(Q_{d+k}) = q^{r\binom{k}{2}} \prod_{i=1}^k \prod_{j=0}^{r-1} \frac{[d+i+j]_q}{[i+j]_q}.
	\end{align*}
	Note that for each $i$,
	\begin{align*}
		\prod_{j=0}^{r-1} \frac{[d+i+j]_q}{[i+j]_q} = \prod_{j=0}^{d-1} \frac{[r+i+j]_q}{[i+j]_q}.
	\end{align*}
	Thus,
	\begin{align*}
		s_{(r^k)}(Q_{d+k}) = q^{r\binom{k}{2}} \prod_{i=1}^k  \prod_{j=0}^{d-1} \frac{[r+i+j]_q}{[i+j]_q} = q^{r\binom{k}{2}} \prod_{j=0}^{d-1} \frac{\qbinom{k+r+j}{k}_q}{\qbinom{k+j}{k}_q},
	\end{align*}
	which yields \eqref{eq:Hoggatt-Schur-q-r} in light of \eqref{eq:q-H-def}.
\end{proof}

\section{Strong $q$-log-concavity of the $q$-analog of $d$-Hoggatt numbers}\label{sec:q-Hoggatt}

\subsection{Proof of Theorem~\ref{thm:main3}}

Let $k,l$ be such that $0<l\le k<n$. It follows from \eqref{eq:Hoggatt-Schur-q} that
\begin{align*}
	\qangle{n}{k}_q \qangle{n}{l}_q - \qangle{n}{k+1}_q \qangle{n}{l-1}_q = s_{(d^k)}(Q_n)s_{(d^l)}(Q_n) - s_{(d^{k+1})}(Q_n)s_{(d^{l-1})}(Q_n).
\end{align*}
Now by Lemma~\ref{le:Schur-rectangular-2}, the right-hand side of the above is in $\mathbb{N}[q]$, thereby implying the strong $q$-log-concavity of $\big( \tqangle{n}{k}_q\big)_{0\le k\le n}$.

\subsection{Proof of Theorem~\ref{thm:main4}}

Let $m,n$ be such that $n\ge m> k$. If $k=0$, we have
\begin{align*}
	\qangle{n}{k}_q \qangle{m}{k}_q - \qangle{n+1}{k}_q \qangle{m-1}{k}_q = 0.
\end{align*}
Now assume $k\ge 1$. It follows from \eqref{eq:Hoggatt-Schur-q-r} that
\begin{align*}
	&\qangle{n}{k}_q \qangle{m}{k}_q - \qangle{n+1}{k}_q \qangle{m-1}{k}_q\\
	&\ \  = q^{(2d-n_\dagger-m_\dagger)\binom{k}{2}} \big(s_{(n_\dagger^k)}(Q_{d+k})s_{(m_\dagger^k)}(Q_{d+k}) - s_{((n_\dagger+1)^k)}(Q_{d+k})s_{((m_\dagger-1)^k)}(Q_{d+k})\big),
\end{align*}
where $m_\dagger := m-k$ and $n_\dagger := n-k$. Now by Lemma~\ref{le:Schur-rectangular-1}, the second factor on the right-hand side of the above is in $\mathbb{N}[q]$. Noting that the left-hand side is in $\mathbb{Z}[q]$ by definition, we conclude that the left-hand side must be in $\mathbb{N}[q]$ because multiplying by the prefactor $q^{(2d-n_\dagger-m_\dagger)\binom{k}{2}}$ does not affect the signs of the coefficients. We then arrive at the strong $q$-log-concavity of $\big( \tqangle{n}{k}_q\big)_{n\ge k}$.

\section{Schur positivity of an auxiliary family of functions}\label{sec:Schur-positivity}

In the next two sections, we set two types of partitions:
\begin{align*}
	R_l := (\underbrace{d,\ldots,d}_{\text{$l$ parts}}), \qquad\qquad P_l^{(j)} := (\underbrace{d,\ldots,d}_{\text{$l$ parts}},j),
\end{align*}
where $0\le j\le d$. Note that $R_l = P_l^{(0)}$ and $R_{l+1} = P_l^{(d)}$. We adopt the convention that all expressions containing $R_{l}$ or $P_{l}^{(j)}$ with $l<0$ are assumed to be $0$. For a partition $\beta$ and an integer $t\ge 0$, define the auxiliary functions
\begin{align}\label{eq:Phi-def}
	\Phi_{t,j}^{\beta} := \sum_{u+v=t} \big(s_{P_u^{(j)}/\beta} s_{R_v} - s_{R_u/\beta} s_{P_v^{(j)}}\big).
\end{align}

The purpose of this section is to establish the following Schur positivity property.

\begin{theorem}\label{th:Phi-Schur-positivity}
	For any $j$ and $t$ with $1\le j\le d-1$ and $t\ge 0$, and any partition $\beta$,
	\begin{align}
		\Phi_{t,j}^{\beta} \ge_s 0.
	\end{align}
\end{theorem}

We shall prove this theorem by induction on $t$. When $t=0$,
\begin{align*}
	\Phi_{0,j}^{\beta} = \begin{cases}
		s_{(j)/\beta}, & \text{if $\beta \ne \varnothing$},\\
		0, & \text{if $\beta = \varnothing$},
	\end{cases}
\end{align*}
which is clearly Schur positive. Now assume that the theorem is true for $0,\ldots,t-1$ with $t\ge 1$. It suffices to show that the coefficients in the Schur expansion of $\Phi_{t,j}^{\beta}$ satisfy
\begin{align}\label{eq:s-alpha-nonnegativity}
	[s_\alpha] \Phi_{t,j}^{\beta} \ge 0
\end{align}
for every partition $\alpha$.

Now there are two cases, depending on whether $\alpha$ contains parts of size $d$ or not. To facilitate our analysis, we introduce a linear map
\begin{align*}
	\begin{array}{cccc}
		\Delta_d \colon & \Lambda & \to & \Lambda\\
		& s_\lambda & \mapsto & s_{d\cup \lambda}
	\end{array}
\end{align*}
with $d\cup \lambda$ the partition derived by adjoining a part of size $d$ in $\lambda$ and rearranging the parts in weakly decreasing order. We will see how adjoining a part $d$ behaves after the projection $\pi_d$.

\begin{lemma}
	For any $F,G\in \Lambda$,
	\begin{align}\label{eq:proj-behavior}
		\pi_d\big(F \Delta_d(G)\big) = \Delta_d\big(\pi_d(FG)\big).
	\end{align}
\end{lemma}

\begin{proof}
	We start with an auxiliary relation that for any $F\in \Lambda$,
	\begin{align}\label{eq:spt-delta-proj}
		\Spomega_d \big(\Delta_d(\pi_d(F))\big) = e_d(X_d) \Spomega_d(F).
	\end{align}
	By the linearity of the operators, it suffices to prove \eqref{eq:spt-delta-proj} for every Schur polynomial $s_\lambda$. If $\lambda_1>d$, we have $\pi_d(s_\lambda) = 0$, and further $\Spomega_d(s_\lambda) = \Spomega_d\big(\pi_d(s_\lambda)\big)=0$ by \eqref{eq:spt-proj}, so that both sides of \eqref{eq:spt-delta-proj} vanishes. Now assume $\lambda_1\le d$. We have
	\begin{align*}
		\Spomega_d \big(\Delta_d(\pi_d(s_\lambda))\big) = \Spomega_d(s_{d\cup \lambda}) = s_{(d\cup \lambda)^\trans}(X_d).
	\end{align*}
	Note that in $d$ variables, there is only one way to fill the first column of $(d\cup \lambda)^\trans$, namely, by filling in $1$ through $n$ from top to bottom. By doing so, we factor out $e_d(X_d) = s_{(1^d)}(X_d)$ and $\Spomega_d(s_\lambda)$, and hence arrive at the desired relation \eqref{eq:spt-delta-proj} for $F=s_\lambda$,
	\begin{align*}
		\Spomega_d \big(\Delta_d(\pi_d(s_\lambda))\big) = e_d(X_d) \Spomega_d(s_\lambda).
	\end{align*}
	Now moving back to \eqref{eq:proj-behavior}, we see that both sides are in the span of $s_\lambda$ with $\lambda_1\le d$.
	Using \eqref{eq:spt-delta-proj} twice and recalling \eqref{eq:spt-product},
	\begin{align*}
		\Spomega_d\big(\Delta_d(\pi_d(FG))\big)&= e_d(X_d) \Spomega_d(FG)\\
		&= e_d(X_d) \Spomega_d(F)\Spomega_d(G)\\
		&= \Spomega_d(F) \Spomega_d\big(\Delta_d(\pi_d(G))\big).
	\end{align*}
	It is clear that $G-\pi_d(G)$ is in $I_d$, and so is $\Delta_d(G)-\Delta_d(\pi_d(G))$. By Lemma~\ref{le:anni}, this difference is annihilated by $\Spomega_d$. Therefore, $\Spomega_d\big(\Delta_d(\pi_d(G))\big) = \Spomega_d\big(\Delta_d(G)\big)$, so that
	\begin{align*}
		\Spomega_d\big(\Delta_d(\pi_d(FG))\big) = \Spomega_d(F) \Spomega_d\big(\Delta_d(G)\big) = \Spomega_d\big(F\Delta_d(G)\big) = \Spomega_d\big(\pi_d(F\Delta_d(G))\big),
	\end{align*}
	where we have used \eqref{eq:spt-product} for the second equality and \eqref{eq:spt-proj} for the last equality. Now the claimed identity \eqref{eq:proj-behavior} follows from the injectivity of $\Spomega_d$ in Lemma~\ref{le:spt-proj}.
\end{proof}

In the following two subsections, we fix $j$ with $1\le j\le d-1$ and put
\begin{align*}
	d' := d-j.
\end{align*}
Before analyzing the two cases of $\alpha$ individually, we record a useful fact about rotated skew diagrams inside a rectangular partition. It is known that the skew Schur function of a diagram is unchanged when the diagram is rotated through $180$ degrees. This is \cite[p.~467, Exerise~7.56(a)]{Sta1999} with its proof given on \cite[p.~516]{Sta1999}.

\begin{lemma}\label{le:rotation}
	Let $\beta\subseteq R_u = (d^u)$ so that $\len(\beta)\le u$, and define
	\begin{align*}
		\kappa := (d-\beta_u,d-\beta_{u-1},\ldots,d-\beta_1).
	\end{align*}
	Then
	\begin{align}\label{eq:rotation}
		s_{R_u/\beta} = s_\kappa,\qquad s_{R_{u+1}/\beta} = s_{d\cup \kappa},\qquad s_{P_u^{(j)}/\beta} = s_{(d\cup \kappa)/(d')}.
	\end{align}
\end{lemma}

\begin{proof}
	Since $\beta$ is inside $R_u$, it is also inside $R_{u+1}$ and $P_u^{(j)}$. Note that $R_u/\beta$ is obtained by rotating $\kappa$ through $180$ degrees, and hence $s_{R_u/\beta} = s_\kappa$. The remaining two relations can be derived in the same way.
\end{proof}

Finally, we introduce two auxiliary coefficients for $u,v\ge 0$ and an arbitrary partition $\lambda$:
\begin{align}\label{eq:AB-def}
	\begin{aligned}
		A_{u,v}(\lambda) &:= \bigangle{s_\lambda,s_{P_u^{(j)}/\beta}s_{R_v}}\ge 0,\\
		B_{u,v}(\lambda) &:= \bigangle{s_\lambda,s_{R_u/\beta}s_{P_v^{(j)}}}\ge 0,
	\end{aligned}
\end{align}
with the nonnegativity coming from the Schur positivity of $s_{P_u^{(j)}/\beta}s_{R_v}$ and $s_{R_u/\beta}s_{P_v^{(j)}}$. In addition, we adopt the convention that $A_{u,v}(\lambda)=B_{u,v}(\lambda)=0$ if $u<0$ or $v<0$.

\begin{lemma}\label{le:AB=0}
	For any partition $\lambda$ with $p_\lambda:=\#\{i\colon \lambda_i>d\}$, we have
	\begin{align}\label{eq:AB=0}
		\begin{cases}
			A_{u,v}(\lambda) = 0, & \text{if $v<p_\lambda$},\\
			B_{u,v}(\lambda) = 0, & \text{if $v<p_\lambda-1$}.
		\end{cases}
	\end{align}
\end{lemma}

\begin{proof}
	By \eqref{eq:Schur-c-2} and \eqref{eq:Schur-c-1},
	\begin{align*}
		s_{P_u^{(j)}/\beta}s_{R_v} = \sum_{\mu} c_{\mu,\beta}^{P_u^{(j)}} s_\mu s_{R_v} = \sum_\lambda \sum_\mu c_{\mu,\beta}^{P_u^{(j)}} c_{\mu,R_v}^{\lambda} s_\lambda,
	\end{align*}
	so that
	\begin{align}\label{eq:A-expression}
		A_{u,v}(\lambda) = \sum_\mu c_{\mu,\beta}^{P_u^{(j)}} c_{\mu,R_v}^{\lambda}.
	\end{align}
	To ensure $A_{u,v}(\lambda)>0$, we need to find a certain $\mu$ such that $c_{\mu,\beta}^{P_u^{(j)}}>0$ and $c_{\mu,R_v}^{\lambda}>0$ hold simultaneously. For the first inequality, Lemma~\ref{le:LR-facts} tells us that $\mu \subseteq P_u^{(j)}$ so that $\mu_1\le d$. For the second inequality, we use \eqref{eq:LR-facts-ineq} in Lemma~\ref{le:LR-facts} to get $\lambda_{v+1} \le \mu_1 + (R_v)_{v+1} = \mu_1$. Thus, $\lambda_{v+1}\le d$. However, if $v<p_\lambda$, then $\lambda_{v+1}>d$ by the definition of $p_\lambda$. This produces a contradiction, and hence $A_{u,v}(\lambda) = 0$ when $v<p_\lambda$. For the second statement involving $B_{u,v}(\lambda)$, we apply a similar analysis but this time to $\lambda_{v+2}$.
\end{proof}

\subsection{The partition $\alpha$ contains parts of size $d$}

Throughout this subsection, since $\alpha$ contains parts of size $d$, we always write it as
\begin{align}
	\alpha := (d+\tau_1,\ldots,d+\tau_p,d,\eta_1,\ldots,\eta_r),
\end{align}
where $\tau:=(\tau_1,\ldots,\tau_p)$ is a partition with exactly $p$ \emph{positive} parts, and $\eta:=(\eta_1,\ldots,\eta_r)$ is a partition with width $\eta_1\le d$. Define
\begin{align*}
	\alpha^- := (d+\tau_1,\ldots,d+\tau_p,\eta_1,\ldots,\eta_r),
\end{align*}
by removing the displayed part of size $d$.

\begin{lemma}\label{le:Hall-alpha}
	For any partition $\lambda$ with $\lambda_1\le d$,
	\begin{align}\label{eq:Hall-alpha}
		\bigangle{s_\lambda,s_{\alpha/P_p^{(j)}}} = \bigangle{s_{\lambda/(d')},s_{\alpha^-/R_p}}.
	\end{align}
\end{lemma}

\begin{proof}
	The skew diagram $\alpha/P_p^{(j)}$ can be split into two disjoint groups, one of shape $\tau$ to the right of the rectangle $R_p$ in the first $p$ rows and the other of shape $(d\cup \eta)/(j)$ below the same rectangle. Moreover, the two groups impose no row or column comparison. Hence,
	\begin{align}\label{eq:alpha&P}
		s_{\alpha/P_p^{(j)}} = s_\tau s_{(d\cup \eta)/(j)}.
	\end{align}
	In a similar vein,
	\begin{align}\label{eq:alpha-&R}
		s_{\alpha^-/R_p} = s_\tau s_\eta.
	\end{align}
	By \eqref{eq:Pieri-skew},
	\begin{align*}
		s_{(d\cup \eta)/(j)} = \sum_{\substack{\xi\colon (d\cup \eta)/\xi\HS\\|d\cup \eta|-|\xi|=j}} s_\xi.
	\end{align*}
	Here, the second condition on $\xi$ is the same as $|\xi|-|\eta| = d-j =d'$. Moreover, the first condition, by \eqref{eq:HS-condition}, is equivalent to $d\ge \xi_1\ge \eta_1\ge \xi_2 \ge \eta_2 \ge \cdots$, and by \eqref{eq:HS-condition} again, to $\xi/\eta$ being a horizontal strip with $\xi_1\le d$. It follows that
	\begin{align*}
		s_{(d\cup \eta)/(j)} = \sum_{\substack{\xi\colon \xi/\eta\HS\\|\xi|-|\eta|=d'\\\xi_1\le d}} s_\xi.
	\end{align*}
	On the other hand, by \eqref{eq:Pieri-product},
	\begin{align*}
		h_{d'}s_\eta = \sum_{\substack{\xi\colon \xi/\eta\HS\\|\xi|-|\eta|=d'}} s_\xi.
	\end{align*}
	We then derive that
	\begin{align*}
		s_{(d\cup \eta)/(j)} \equiv h_{d'}s_\eta \pmod{I_d},
	\end{align*}
	and further from the fact that $I_d$ is an ideal stated in Lemma~\ref{le:I-ideal} that
	\begin{align*}
		s_\tau s_{(d\cup \eta)/(j)} \equiv h_{d'}s_\tau s_\eta \pmod{I_d}.
	\end{align*}
	Recalling the assumption $\lambda_1\le d$ and using the orthogonality of the Hall inner product, we have
	\begin{align}\label{eq:Hall-from-I}
		\bigangle{s_\lambda,s_\tau s_{(d\cup \eta)/(j)}} = \bigangle{s_\lambda,h_{d'}s_\tau s_\eta}.
	\end{align}
	Finally, noting that $h_{d'} = s_{(d')}$ in \eqref{eq:h/e-def}, the claimed relation follows from
	\begin{align*}
		\bigangle{s_\lambda,s_{\alpha/P_p^{(j)}}} &= \bigangle{s_{\lambda},s_\tau s_{(d\cup \eta)/(j)}}\\
		&= \bigangle{s_\lambda,h_{d'}s_\tau s_\eta}\\
		&= \bigangle{s_{\lambda/(d')},s_\tau s_\eta}\\
		&= \bigangle{s_{\lambda/(d')},s_{\alpha^-/R_p}},
	\end{align*}
	where we have used \eqref{eq:alpha&P}, \eqref{eq:Hall-from-I}, \eqref{eq:Hall-F}, and \eqref{eq:alpha-&R} for the four equalities, respectively.
\end{proof}

We also need a King--Tollu--Toumazet type theorem discovered by Cho, Jung, and Moon~\cite[p.~487, Theorem~1.7 with $r=1$]{CJM2008}.

\begin{lemma}\label{le:CJM}
	Let $\lambda,\mu,\nu$ be partitions. Assume that $c_{\mu,\nu}^\lambda > 0$ and that there are indices $i,j,k$ such that $i+j=k+1$ and $\lambda_k = \mu_i+\nu_j$. Then
	\begin{align}
		c_{\mu,\nu}^{\lambda} = c_{\mu',\nu'}^{\lambda'},
	\end{align}
	where $\lambda'$, $\mu'$, and $\nu'$ are obtained by removing $\lambda_k$, $\mu_i$, and $\nu_j$ from $\lambda$, $\mu$, and $\nu$, respectively.
\end{lemma}

Now we may prove the desired nonnegativity \eqref{eq:s-alpha-nonnegativity} by combining the following result and our inductive assumption $\Phi_{t-1,j}^\beta\ge_s 0$ so that $[s_{\alpha^-}]\Phi_{t-1,j}^\beta \ge 0$.

\begin{theorem}\label{th:alpha-contain-d}
	Let $\alpha$ be a partition containing parts of size $d$. Then
	\begin{align}\label{eq:Phi-ineq-alpha-d}
		[s_{\alpha}]\Phi_{t,j}^\beta\ge [s_{\alpha^-}]\Phi_{t-1,j}^\beta.
	\end{align}
\end{theorem}

Before proceeding with the proof of this theorem, we establish the following relations involving $A_{u,v}(\alpha)$ and $B_{u,v}(\alpha)$.

\begin{lemma}
	We have
	\begin{align}\label{eq:A-alpha}
		A_{u,v}(\alpha) = \begin{cases}
			A_{u,v-1}(\alpha^-), & \text{if $v\ge p+1$},\\
			B_{u,v}(\alpha), & \text{if $v=p$}.
		\end{cases}
	\end{align}
	Also,
	\begin{align}\label{eq:B-alpha}
		B_{u,v}(\alpha) = \begin{cases}
			B_{u,v-1}(\alpha^-), & \text{if $v\ge p+1$},\\
			A_{u,v}(\alpha), & \text{if $v=p$},
		\end{cases}
	\end{align}
	and
	\begin{align}\label{eq:B-alpha-ineq}
		B_{u,p-1}(\alpha) =\begin{cases}
			B_{u-1,p-1}(\alpha^-), & \text{if $B_{u,p-1}(\alpha)>0$},\\
			0, & \text{otherwise}.
		\end{cases}
	\end{align}
\end{lemma}

\begin{proof}
	We first consider the case $v\ge p+1$. In this situation, the skew diagrams $\alpha/R_v$ and $\alpha^-/R_{v-1}$ are the same. Using \eqref{eq:Hall-F} twice, we have
	\begin{align*}
		A_{u,v}(\alpha) = \bigangle{s_{\alpha/R_v},s_{P_u^{(j)}/\beta}} = \bigangle{s_{\alpha^-/R_{v-1}},s_{P_u^{(j)}/\beta}} = A_{u,v-1}(\alpha^-).
	\end{align*}
	For $B_{u,v}(\alpha)$, we only need the fact that the skew diagrams $\alpha/P_v^{(j)}$ and $\alpha^-/P_{v-1}^{(j)}$ are the same when $v\ge p+1$.
	
	Next, we look at the case $v=p$. First, assume $\beta\not\subseteq R_u$. Then $B_{u,p}(\alpha)=0$ because $R_u/\beta$ is not a valid skew diagram. If we further have $\beta\not\subseteq P_u^{(j)}$, then it is also true that $A_{u,p}(\alpha)=0$. If $\beta\subseteq P_u^{(j)}$, then $\len(\beta)=u+1$ and $1\le \beta_{u+1}\le j$. Let $\kappa' := (d-\beta_{u+1},\ldots,d-\beta_1)$ be the rotated completement of $\beta$ in $R_{u+1}$, so that $\kappa'_1<d$. The same arguments as we show \eqref{eq:rotation} give $s_{P_u^{(j)}/\beta} = s_{\kappa'/(d')}$. Similar to \eqref{eq:A-expression}, we have
	\begin{align*}
		A_{u,p}(\alpha) = \sum_{\mu} c_{\mu,(d')}^{\kappa'} c_{\mu,R_p}^{\alpha}.
	\end{align*}
	If $A_{u,p}(\alpha)>0$, then there is a certain $\mu$ such that $c_{\mu,(d')}^{\kappa'}>0$ and $c_{\mu,R_p}^{\alpha}>0$, which, as in the proof of Lemma~\ref{le:AB=0}, imply that
	\begin{align*}
		\alpha_{p+1} \le \mu_1 + (R_p)_{p+1} = \mu_1 \le \kappa'_1 < d.
	\end{align*}
	However, this is impossible because $\alpha_{p+1}=d$. Thus, we still have $A_{u,p}(\alpha)=0$. The above arguments tell us that when $\beta\not\subseteq R_u$,
	\begin{align*}
		A_{u,p}(\alpha)=B_{u,p}(\alpha)=0.
	\end{align*}
	Now assume $\beta\subseteq R_u$. Let $\kappa$ be as in Lemma~\ref{le:rotation} and put $\kappa^+:=d\cup \kappa$. In light of \eqref{eq:Hall-F} and \eqref{eq:rotation},
	\begin{align*}
		A_{u,p}(\alpha) = \bigangle{s_\alpha,s_{P_u^{(j)}/\beta}s_{R_p}} = \bigangle{s_{\alpha/R_p},s_{P_u^{(j)}/\beta}} = \bigangle{s_{\alpha/R_p},s_{\kappa^+/(d')}}.
	\end{align*}
	In the way as we show \eqref{eq:alpha&P}, it is true that
	\begin{align*}
		s_{\alpha/R_p} = s_{\tau} s_{d\cup\eta} = \sum_{\lambda} c_{\tau,d\cup\eta}^{\lambda} s_\lambda.
	\end{align*}
	By the Littlewood--Richardson rule, we may impose the restriction $d\cup \eta\subseteq \lambda$ to $\lambda$ so that $\lambda_1\ge d$. Thus,
	\begin{align*}
		s_{\alpha/R_p} = \sum_{\lambda\colon \lambda_1=d} c_{\tau,d\cup\eta}^{\lambda} s_\lambda + \sum_{\lambda\colon \lambda_1>d} c_{\tau,d\cup\eta}^{\lambda} s_\lambda.
	\end{align*}
	In particular, we have, with \eqref{eq:proj-behavior} used for the second equality and \eqref{eq:alpha-&R} used for the last equality,
	\begin{align*}
		\sum_{\lambda\colon \lambda_1=d} c_{\tau,d\cup\eta}^{\lambda} s_\lambda = \pi_d(s_{\tau} s_{d\cup\eta}) = \Delta_d\big(\pi_d(s_\tau s_\eta)\big) = \Delta_d\big(\pi_d(s_{\alpha^-/R_p})\big).
	\end{align*}
	In the meantime, by \eqref{eq:Pieri-skew},
	\begin{align*}
		s_{\kappa^+/(d')} = \sum_{\substack{\delta\colon \kappa^+/\delta\HS\\|\kappa^+|-|\delta|=d'}} s_\delta.
	\end{align*}
	Since $\kappa^+_1= d$, we may split the above as
	\begin{align*}
		s_{\kappa^+/(d')} = \sum_{\substack{\delta\colon \kappa^+/\delta\HS\\|\kappa^+|-|\delta|=d'\\\delta_1=d}} s_\delta + \sum_{\substack{\delta\colon \kappa^+/\delta\HS\\|\kappa^+|-|\delta|=d'\\\delta_1<d}} s_\delta.
	\end{align*}
	In particular, we have, with \eqref{eq:Pieri-skew} recalled,
	\begin{align*}
		\sum_{\substack{\delta\colon \kappa^+/\delta\HS\\|\kappa^+|-|\delta|=d'\\\delta_1=d}} s_\delta = \sum_{\substack{\delta'\colon \kappa/\delta'\HS\\|\kappa|-|\delta'|=d'}} s_{d\cup\delta'} = \Delta_d(s_{\kappa/(d')}).
	\end{align*}
	Thus,
	\begin{align*}
		A_{u,p}(\alpha) = \left\langle\sum_{\lambda\colon \lambda_1=d} c_{\tau,d\cup\eta}^{\lambda} s_\lambda,\sum_{\substack{\delta\colon \kappa^+/\delta\HS\\|\kappa^+|-|\delta|=d'\\\delta_1=d}} s_\delta\right\rangle = \bigangle{\Delta_d\big(\pi_d(s_{\alpha^-/R_p})\big),\Delta_d(s_{\kappa/(d')})}.
	\end{align*}
	In the above Hall inner product, it is safe to remove the operator $\Delta_d$ because for any two Schur polynomials $s_\mu$ and $s_\nu$, it is true that
	\begin{align*}
		\bigangle{s_\mu,s_\nu} = \bigangle{s_{d\cup\mu},s_{d\cup\nu}} = \bigangle{\Delta_d(s_\mu),\Delta_d(s_\nu)}.
	\end{align*}
	Also, we may remove the operator $\pi_d$ in the first term because if a certain $s_\delta$ appears in the Schur expansion of $s_{\kappa/(d')}$ in the second term, then by \eqref{eq:Pieri-skew}, we must have $\delta\subseteq \kappa$ so that $\delta_1\le \kappa_1\le d$. Therefore,
	\begin{align*}
		A_{u,p}(\alpha) = \bigangle{s_{\alpha^-/R_p},s_{\kappa/(d')}} = \bigangle{s_{\alpha/P_p^{(j)}},s_\kappa} = \bigangle{s_{\alpha/P_p^{(j)}},s_{R_u/\beta}},
	\end{align*}
	where we have used \eqref{eq:Hall-alpha} for the second equality and \eqref{eq:rotation} for the last equality. We then derive from \eqref{eq:Hall-F} the desired relation
	\begin{align*}
		A_{u,p}(\alpha) = B_{u,p}(\alpha).
	\end{align*}
	
	Finally, we work out the case $v=p-1$ for $B_{u,v}(\alpha)$, namely, the relation \eqref{eq:B-alpha-ineq}. If $p=0$, then there is nothing to prove. Now assume $p\ge 1$ so that $\alpha \ne P_{p-1}^{(j)}$. If $u=0$, then $B_{0,p-1}(\alpha)=\bigangle{s_\alpha,s_{P_{p-1}^{(j)}}}=0$. Now we further assume $u\ge 1$. Here we only need to look at the case $B_{u,p-1}(\alpha)>0$, which requires $\beta\subseteq R_u$. Let $\kappa$ be as in Lemma~\ref{le:rotation}. In light of \eqref{eq:rotation} and \eqref{eq:Schur-c-1},
	\begin{align*}
		B_{u,p-1}(\alpha) = \bigangle{s_\alpha,s_{R_u/\beta}s_{P_{p-1}^{(j)}}} = \bigangle{s_\alpha,s_\kappa s_{P_{p-1}^{(j)}}} = c_{\kappa,P_{p-1}^{(j)}}^{\alpha}.
	\end{align*}
	It follows from \eqref{eq:LR-facts-ineq} that
	\begin{align*}
		d = \alpha_{p+1} \le \kappa_1+(P_{p-1}^{(j)})_{p+1}=\kappa_1\le d.
	\end{align*}
	Therefore, $\kappa_1=d$. Remove this first part from $\kappa$ and call the resulting partition $\kappa^-$. Since
	\begin{align*}
		\alpha_{p+1} = d = d+0 = \kappa_1 + (P_{p-1}^{(j)})_{p+1},
	\end{align*}
	we derive from Lemma~\ref{le:CJM} that
	\begin{align*}
		B_{u,p-1}(\alpha) = c_{\kappa,P_{p-1}^{(j)}}^{\alpha} = c_{\kappa^-,P_{p-1}^{(j)}}^{\alpha^-}.
	\end{align*}
	Note that $\kappa_1=d$ implies that $\beta_u=0$ so that $\kappa^-$ is the rotated completement of $\beta$ in $R_{u-1}$. We then use \eqref{eq:rotation} and \eqref{eq:Schur-c-1} to get
	\begin{align*}
		B_{u-1,p-1}(\alpha^-) = \bigangle{s_{\alpha^-},s_{R_{u-1}/\beta}s_{P_{p-1}^{(j)}}} = \bigangle{s_{\alpha^-},s_{\kappa-} s_{P_{p-1}^{(j)}}} = c_{\kappa^-,P_{p-1}^{(j)}}^{\alpha^-}.
	\end{align*}
	Thus,
	\begin{align*}
		B_{u,p-1}(\alpha) = B_{u-1,p-1}(\alpha^-),
	\end{align*}
	as claimed.
\end{proof}

Now we complete the proof of Theorem~\ref{th:alpha-contain-d}.

\begin{proof}[Proof of Theorem~\ref{th:alpha-contain-d}]
	By the definitions in \eqref{eq:AB-def}, we have
	\begin{align*}
		[s_\alpha] \Phi_{t,j}^\beta = \sum_{v\ge 0} \big(A_{t-v,v}(\alpha) - B_{t-v,v}(\alpha)\big),
	\end{align*}
	and
	\begin{align*}
		[s_{\alpha^-}] \Phi_{t-1,j}^\beta = \sum_{v\ge 0} \big(A_{t-v,v-1}(\alpha^-) - B_{t-v,v-1}(\alpha^-)\big).
	\end{align*}
	Invoking \eqref{eq:AB=0}, \eqref{eq:A-alpha}, and \eqref{eq:B-alpha},
	\begin{align*}
		[s_\alpha] \Phi_{t,j}^\beta - [s_{\alpha^-}] \Phi_{t-1,j}^\beta = B_{t-p,p-1}(\alpha^-) - B_{t-p+1,p-1}(\alpha) \ge 0,
	\end{align*}
	where we have used \eqref{eq:B-alpha-ineq} for the nonnegativity. Thus, the desired inequality \eqref{eq:Phi-ineq-alpha-d} is true.
\end{proof}

\subsection{The partition $\alpha$ contains no part of size $d$}

Throughout this subsection, since $\alpha$ does not contain any part of size $d$, we always write it as
\begin{align}
	\alpha := (d+\tau_1,\ldots,d+\tau_p,\rho_1,\ldots,\rho_r),
\end{align}
where $\tau:=(\tau_1,\ldots,\tau_p)$ is a partition with exactly $p$ \emph{positive} parts, and $\rho:=(\rho_1,\ldots,\rho_r)$ is a partition with width $\rho_1< d$. For convenience, we put
\begin{align*}
	\tau^{(d)} := (d+\tau_1,\ldots,d+\tau_p).
\end{align*}
Define three auxiliary skew diagrams:
\begin{align*}
	D:=\alpha/P_{p-1}^{(j)},\qquad T:=\tau^{(d)}/P_{p-1}^{(j)},\qquad L:=((d-1)\cup\rho)/(j-1).
\end{align*}

We first prove two useful results about the skew diagrams $L$ and $T$.

\begin{lemma}
	We have
	\begin{align}\label{eq:s_L}
		\pi_d(h_{d'}s_\rho-s_L) = \Delta_d(s_{\rho/(j)}).
	\end{align}
\end{lemma}

\begin{proof}
	Similar arguments as in the proof of Lemma~\ref{le:Hall-alpha} tell us that
	\begin{align*}
		h_{d'}s_\rho = \sum_{\substack{\xi\colon \xi/\rho\HS\\|\xi|-|\rho|=d'}} s_\xi,\qquad\qquad s_L = \sum_{\substack{\xi\colon \xi/\rho\HS\\|\xi|-|\rho|=d'\\\xi_1\le d-1}} s_\xi,
	\end{align*}
	so that
	\begin{align*}
		\pi_d(h_{d'}s_\rho-s_L) = \sum_{\substack{\xi\colon \xi/\rho\HS\\|\xi|-|\rho|=d'\\\xi_1= d}} s_\xi = \sum_{\substack{\xi'\colon \rho/\xi'\HS\\|\rho|-|\xi'|=j}} s_{d\cup \xi'},
	\end{align*}
	where we have written $\xi = d\cup \xi'$ and applied the condition \eqref{eq:HS-condition} for the second equality. The claimed identity then follows from \eqref{eq:Pieri-skew}.
\end{proof}

\begin{lemma}
	We have
	\begin{align}\label{eq:s_T}
		h_{d'}s_\tau - s_T = \sum_{\substack{\zeta\colon \zeta/\tau\HS\\|\zeta|-|\tau|=d'\\\len(\zeta)=p+1}} s_\zeta.
	\end{align}
\end{lemma}

\begin{proof}
	First, we know from \eqref{eq:Pieri-product} that
	\begin{align*}
		h_{d'}s_\tau = \sum_{\substack{\zeta\colon \zeta/\tau\HS\\|\zeta|-|\tau|=d'\\\len(\zeta)\le p+1}} s_\zeta,
	\end{align*}
	where we require $\len(\zeta)\le p+1$ because otherwise there are at least two cells in the first column of the horizontal strip $\zeta/\tau$, which is impossible. Next, by \eqref{eq:Schur-c-3},
	\begin{align*}
		s_T = \sum_{\zeta\colon \zeta\subseteq\tau^{(d)}} c_{P_{p-1}^{(j)},\zeta}^{\tau^{(d)}} s_\zeta.
	\end{align*}
	To ensure the positivity of the Littlewood--Richardson coefficients in the above, we must have $\len(\zeta)\le p$ and
	\begin{align*}
		|\zeta| = |\tau^{(d)}| - |P_{p-1}^{(j)}| = |\tau| + d'.
	\end{align*}
	Note also that the skew diagram $\tau^{(d)}/P_{p-1}^{(j)}$ can be split into two disjoint groups, one of shape $\tau$ to the right of the rectangle $R_p$, and the other consisting of $d'$ cells from the $(j+1)$-th column to the $d$-th column in the $p$-th row. Moreover, in a Littlewood--Richardson tableau of this shape and type $\zeta$, both groups are uniquely filled out: for the first group, according to the first assertion in Lemma~\ref{le:LR-facts} together with the column strictness, it is only possible to fill all entries in the $i$-th row with $i$ for each $i$ where $1\le i\le p$; for the second group, the row weak monotonicity and the type $\zeta$ of the tableau require that the cells are filled with $1$ of multiplicity $\zeta_1-\tau_1$, followed by $2$ of multiplicity $\zeta_2-\tau_2$, and so forth, till $p$ of multiplicity $\zeta_p-\tau_p$. In particular, by the Littlewood--Richardson condition, we see that for each $i$,
	\begin{align*}
		\tau_{i-1}-\tau_i\ge \zeta_i-\tau_i\ge 0,
	\end{align*}
	where we put $\tau_0 = \infty$. Thus,
	\begin{align*}
		\zeta_1\ge \tau_1\ge \zeta_2\ge \tau_2\ge \cdots,
	\end{align*}
	meaning that $\zeta/\tau$ is a horizontal strip by \eqref{eq:HS-condition}. To summarize the above discussions, we have
	\begin{align*}
		c_{P_{p-1}^{(j)},\zeta}^{\tau^{(d)}} = \begin{cases}
			1, & \text{if $\zeta/\tau\HS$, $|\zeta|-|\tau|=d'$, and $\len(\zeta)\le p$},\\
			0, & \text{otherwise}.
		\end{cases}
	\end{align*}
	It follows that
	\begin{align*}
		s_T = \sum_{\substack{\zeta\colon \zeta/\tau\HS\\|\zeta|-|\tau|=d'\\\len(\zeta)\le p}} s_\zeta,
	\end{align*}
	which gives \eqref{eq:s_T} by invoking the expression of $h_{d'}s_\tau$ derived at the beginning of this proof.
\end{proof}

In the next two lemmas, we shall use the  \emph{dual Jacobi--Trudi formula}~\cite[p.~344, Corollary~7.16.2]{Sta1999}, which states that for two partitions $\lambda$ and $\mu$ with length at most $l$ such that $\mu\subseteq \lambda$,
\begin{align}\label{dq:dualJT}
	s_{(\lambda/\mu)^\trans}(X_N) = \underset{1\le a,b\le l}{\det}\big(e_{\lambda_a-\mu_b-a+b}(X_N)\big).
\end{align}

\begin{lemma}
	Let $N=d$ or $d-1$. We have
	\begin{align}\label{eq:s_D}
		\pi_N(s_D) = \pi_N(s_\rho s_T).
	\end{align}
\end{lemma}

\begin{proof}
	We apply \eqref{dq:dualJT} to $D^\trans = (\alpha/P_{p-1}^{(j)})^\trans$ with $l=p+r$. In particular, when $1\le a\le p$ and $p+1\le b\le p+r$, the entry in the dual Jacobi--Trudi matrix is $e_{a+\tau_a-a+b}(X_N)$ with the index $a+\tau_a-a+b > N$ so that $e_{a+\tau_a-a+b}(X_N)=0$. Now the dual Jacobi--Trudi matrix has the block decomposition
	\begin{align*}
		\begin{pmatrix}
			\mathbf{A} & \mathbf{0}\\
			\star & \mathbf{B}
		\end{pmatrix},
	\end{align*}
	where
	\begin{align*}
		\mathbf{A} := \big(e_{\tau^{(d)}_a-(P_{p-1}^{(j)})_b-a+b}(X_N)\big)_{1\le a,b\le p},\qquad \mathbf{B}:=\big(e_{\rho_{a'}-a'+b'}(X_N)\big)_{1\le a',b'\le r}.
	\end{align*}
	Thus,
	\begin{align*}
		s_{D^\trans}(X_N) = \det \mathbf{A} \det \mathbf{B} = s_{T^\trans}(X_N) s_{\rho^\trans}(X_N),
	\end{align*}
	where we have used \eqref{dq:dualJT} again for $\det\mathbf{A}$ and $\det\mathbf{B}$. In light of \eqref{eq:omega-product} and \eqref{eq:omega-skew}, we have
	\begin{align*}
		\Spomega_N(s_D) = \Spomega_N(s_\rho s_T),
	\end{align*}
	which, according to \eqref{eq:spt-proj}, further gives
	\begin{align*}
		\Spomega_N\big(\pi_N(s_D)\big) = \Spomega_N\big(\pi_N(s_\rho s_T)\big).
	\end{align*}
	Because of the injectivity of $\Spomega_N$ on the span of $s_\lambda$ with $\lambda_1\le N$ as stated in Lemma~\ref{le:spt-proj}, we conclude the desired relation.
\end{proof}

\begin{lemma}
	We have
	\begin{align}\label{eq:s_gamma/delta}
		s_{\tau^\trans}(X_d) s_{L^\trans}(X_d) - s_{\rho^\trans} s_{T^\trans}(X_d) = s_{(\gamma/\delta)^\trans}(X_d),
	\end{align}
	where
	\begin{align*}
		\gamma:= (d-2+\tau_1,\ldots,d-2+\tau_p,d-1,\rho_1,\ldots,\rho_r),
	\end{align*}
	and
	\begin{align*}
		\delta:=(\underbrace{d-2,\ldots,d-2}_{\text{$p$ parts}},j-1).
	\end{align*}
\end{lemma}

\begin{proof}
	We apply \eqref{dq:dualJT} to $(\gamma/\delta)^\trans$ with $l=p+r+1$. It is easy to verify that the associated dual Jacobi--Trudi matrix has the following block decomposition:
	\begin{align*}
		\begin{pmatrix}
			\mathbf{M} & \mathbf{k} & \mathbf{0}\\
			(0,\ldots,0,1) & e_{d'}(X_d) & (e_d(X_d),0,\ldots,0)\\
			\mathbf{0} & \mathbf{l} & \mathbf{N}
		\end{pmatrix},
	\end{align*}
	where
	\begin{align*}
		\mathbf{M} := \big(e_{\tau_a-a+b}(X_d)\big)_{1\le a,b\le p},\qquad\qquad \mathbf{k} := \big(e_{\tau_a-a+p+d'}(X_d)\big)_{1\le a\le p}^\trans,
	\end{align*}
	and
	\begin{align*}
		\mathbf{N} := \big(e_{\rho_{a'}-a'+b'}(X_d)\big)_{1\le a',b'\le r},\qquad\qquad \mathbf{l} := \big(e_{\rho_{a'}-a'-j+1}(X_d)\big)_{1\le a'\le p}^\trans.
	\end{align*}
	Expanding the determinant gives
	\begin{align*}
		&s_{(\gamma/\delta)^\trans}(X_d)\\
		&\qquad= \det \mathbf{M} \det \begin{pmatrix}
			e_{d'}(X_d) & (e_d(X_d),0,\ldots,0)\\
			\mathbf{l} & \mathbf{N}
		\end{pmatrix} - \det \mathbf{N} \det(\mathbf{M}_1,\ldots,\mathbf{M}_{p-1},\mathbf{k}).
	\end{align*}
	Using \eqref{dq:dualJT} again, we find that
	\begin{align*}
		s_{\tau^\trans}(X_d) = \det \mathbf{M},\qquad\qquad s_{\rho^\trans} = \det \mathbf{N}.
	\end{align*}
	Also,
	\begin{align*}
		s_{L^\trans}(X_d) = s_{((d-1)\cup\rho)/(j-1)^\trans}(X_d) = \det \begin{pmatrix}
			e_{d'}(X_d) & (e_d(X_d),0,\ldots,0)\\
			\mathbf{l} & \mathbf{N}
		\end{pmatrix},
	\end{align*}
	and
	\begin{align*}
		s_{T^\trans}(X_d) = s_{(\tau^{(d)}/P_{p-1}^{(j)})^\trans}(X_d) = \det(\mathbf{M}_1,\ldots,\mathbf{M}_{p-1},\mathbf{k}).
	\end{align*}
	The desired identity \eqref{eq:s_gamma/delta} then follows.
\end{proof}

\begin{corollary}
	We have
	\begin{align}\label{eq:tau*L-rho*T}
		\pi_d(s_\tau s_L - s_\rho s_T) \ge_s 0.
	\end{align}
\end{corollary}

\begin{proof}
	It follows from \eqref{eq:s_gamma/delta} that
	\begin{align*}
		\Spomega_d(s_\tau s_L - s_\rho s_T) = s_{(\gamma/\delta)^\trans}(X_d) \ge_s 0.
	\end{align*}
	Note that $(s_\tau s_L - s_\rho s_T) - \pi_d(s_\tau s_L - s_\rho s_T)$ is in $I_d$ so that it is annihilated by $\Spomega_d$ according to Lemma~\ref{le:anni}. Thus,
	\begin{align*}
		\Spomega_d\big(\pi_d(s_\tau s_L - s_\rho s_T)\big) = \Spomega_d(s_\tau s_L - s_\rho s_T) \ge_s 0.
	\end{align*}
	Since on the span of $s_\lambda$ with $\lambda_1\le N$, the injective map $\Spomega_d$ sends $s_\lambda$ bijectively to $s_{\lambda^\trans}(X_d)$, the Schur positivity of $\pi_d(s_\tau s_L - s_\rho s_T)$ is then immediate.
\end{proof}

The final piece of the ingredients we need is a simplification of $[s_\alpha] \Phi_{t,j}^\beta$ for our choice of $\alpha$.

\begin{lemma}
	We have
	\begin{align}\label{eq:alpha-no-d-Phi-simple}
		[s_\alpha] \Phi_{t,j}^\beta = \begin{cases}
			A_{t-p,p}(\alpha) - B_{t-p,p}(\alpha) - B_{t-p+1,p-1}(\alpha), & \text{if $t\ge p$},\\
			0, & \text{if $t<p$}.
		\end{cases}
	\end{align}
\end{lemma}

\begin{proof}
	Recalling \eqref{eq:AB-def},
	\begin{align*}
		[s_\alpha] \Phi_{t,j}^\beta = \sum_{v\ge 0} \big(A_{t-v,v}(\alpha) - B_{t-v,v}(\alpha)\big).
	\end{align*}
	According to \eqref{eq:A-expression}, if $A_{t-v,v}(\alpha)>0$, then we must have $R_v\subseteq \alpha$, so that $\alpha_v\ge d$. Since $d$ does not appear in $\alpha$, we indeed have $\alpha_v>d$, which implies that $v\le p$. By the relation for $A$ in \eqref{eq:AB=0},
	\begin{align*}
		\sum_{v\ge 0} A_{t-v,v}(\alpha) = A_{t-p,p}(\alpha).
	\end{align*}
	Similarly, if $B_{t-v,v}(\alpha)>0$, we need $P_v^{(j)}\subseteq \alpha$, which in the same way, gives $v\le p$. By the relation for $B$ in \eqref{eq:AB=0},
	\begin{align*}
		\sum_{v\ge 0} B_{t-v,v}(\alpha) = B_{t-p,p}(\alpha) + B_{t-p+1,p-1}(\alpha).
	\end{align*}
	Hence,
	\begin{align*}
		[s_\alpha] \Phi_{t,j}^\beta = A_{t-p,p}(\alpha) - B_{t-p,p}(\alpha) - B_{t-p+1,p-1}(\alpha).
	\end{align*}
	Now \eqref{eq:alpha-no-d-Phi-simple} is plain when $t\ge p$ or $t\le p-2$. For the last case $t=p-1$, we have
	\begin{align*}
		[s_\alpha] \Phi_{p-1,j}^\beta = - B_{0,p-1}(\alpha) = -\bigangle{s_\alpha,s_{\varnothing/\beta}s_{P_{p-1}^{(j)}}}.
	\end{align*}
	If $\beta\ne \varnothing$, then the skew diagram $\varnothing/\beta$ is invalid so that $s_{\varnothing/\beta}=0$ and hence that $[s_\alpha] \Phi_{p-1,j}^\beta = 0$. If $\beta=\varnothing$, then $[s_\alpha] \Phi_{p-1,j}^\varnothing = -\bigangle{s_\alpha,s_{P_{p-1}^{(j)}}}$. Note that $\alpha\ne P_{p-1}^{(j)}$; otherwise, $\alpha$ has a part of size $d$ because $p-1=t\ge 1$ by the inductive assumption, but this violates our requirement for $\alpha$. Hence, $[s_\alpha] \Phi_{p-1,j}^\varnothing$ still vanishes.
\end{proof}

Unlike Theorem~\ref{th:alpha-contain-d} for the first case of $\alpha$ in the preceding subsection, this time we confirm \eqref{eq:s-alpha-nonnegativity} directly.

\begin{theorem}\label{th:alpha-contain-no-d}
	Let $\alpha$ be a partition containing no part of size $d$. Then
	\begin{align}\label{eq:Phi-ineq-alpha-no-d}
		[s_{\alpha}]\Phi_{t,j}^\beta\ge 0.
	\end{align}
\end{theorem}

\begin{proof}
	It suffices to assume $t\ge p$ in light of \eqref{eq:alpha-no-d-Phi-simple}. By the same arguments as we show \eqref{eq:alpha&P} and \eqref{eq:alpha-&R}, it is true that
	\begin{align*}
		s_{\alpha/R_p} = s_\tau s_\rho,
	\end{align*}
	and
	\begin{align*}
		s_{\alpha/P_p^{(j)}} = s_\tau s_{\rho/(j)}.
	\end{align*}
	We also recall from the definition of the skew diagram $D$ that
	\begin{align*}
		s_{\alpha/P_{p-1}^{(j)}} = s_D.
	\end{align*}
	For convenience, write
	\begin{align*}
		t' := t-p.
	\end{align*}
	Note that
	\begin{align*}
		R_{t'} \subseteq P_{t'}^{(j)} \subseteq R_{t'+1}.
	\end{align*}
	We then have four cases of $\beta$.
	
	\textit{Case 1: $\beta\subseteq R_{t'}$.} Let $\kappa:=(d-\beta_{t'},\ldots,d-\beta_1)$ be the rotated completement of $\beta$ in $R_{t'}$ and put $\kappa^+:=d \cup \kappa$. By \eqref{eq:rotation},
	\begin{align*}
		s_{R_{t'}/\beta} = s_\kappa,\qquad s_{R_{t'+1}/\beta} = s_{\kappa^+},\qquad s_{P_{t'}^{(j)}/\beta} = s_{\kappa^+/(d')}.
	\end{align*}
	Now,
	\begin{align*}
		A_{t',p}(\alpha) = \bigangle{s_\alpha,s_{P_{t'}^{(j)}/\beta}s_{R_p}}= \bigangle{s_{\alpha/R_p},s_{P_{t'}^{(j)}/\beta}}= \bigangle{s_\tau s_\rho,s_{\kappa^+/(d')}}= \bigangle{s_{\kappa^+},h_{d'}s_\tau s_\rho},
	\end{align*}
	where we have utilized \eqref{eq:Hall-F} for the second and fourth equalities. Since $\kappa_1^+\le d$, we obtain
	\begin{align*}
		A_{t',p}(\alpha) = \bigangle{s_{\kappa^+},\pi_d(h_{d'}s_\tau s_\rho)}.
	\end{align*}
	Next,
	\begin{align*}
		B_{t',p}(\alpha) = \bigangle{s_\alpha,s_{R_{t'}/\beta}s_{P_{p}^{(j)}}}= \bigangle{s_{\alpha/P_{p}^{(j)}},s_{R_{t'}/\beta}}= \bigangle{s_\tau s_{\rho/(j)},s_{\kappa}}.
	\end{align*}
	Since $\kappa_1\le d$, we further have
	\begin{align*}
		B_{t',p}(\alpha) = \bigangle{s_{\kappa},\pi_d(s_\tau s_{\rho/(j)})} = \bigangle{s_{\kappa^+},\Delta_d\big(\pi_d(s_\tau s_{\rho/(j)})\big)}.
	\end{align*}
	We see that $\Delta_d\big(\pi_d(s_\tau s_{\rho/(j)})\big)$ is equal to $\pi_d\big(s_\tau\Delta_d(s_{\rho/(j)})\big)$ by \eqref{eq:proj-behavior}, and further to $\pi_d\big(s_\tau\pi_d(h_{d'}s_\rho-s_L)\big)$ by \eqref{eq:s_L}, and lastly to $\pi_d\big(s_\tau(h_{d'}s_\rho-s_L)\big)$ by \eqref{eq:FG-proj}. Therefore,
	\begin{align*}
		B_{t',p}(\alpha) = \bigangle{s_{\kappa^+},\pi_d\big(s_\tau(h_{d'}s_\rho-s_L)\big)}.
	\end{align*}
	Finally,
	\begin{align*}
		B_{t'+1,p-1}(\alpha) = \bigangle{s_\alpha,s_{R_{t'+1}/\beta}s_{P_{p-1}^{(j)}}} = \bigangle{s_{\alpha/P_{p-1}^{(j)}},s_{R_{t'+1}/\beta}} = \bigangle{s_{\kappa^+},s_D}.
	\end{align*}
	Once again, we can restrict $s_D$ to $\pi_d(s_D)$ because $\kappa_1^+\le d$. Hence,
	\begin{align*}
		B_{t'+1,p-1}(\alpha) = \bigangle{s_{\kappa^+},\pi_d(s_D)}.
	\end{align*}
	It follows from \eqref{eq:alpha-no-d-Phi-simple} that
	\begin{align*}
		[s_\alpha] \Phi_{t,j}^\beta &= \bigangle{s_{\kappa^+},\pi_d(h_{d'}s_\tau s_\rho)-\pi_d\big(s_\tau(h_{d'}s_\rho-s_L)\big)-\pi_d(s_D)}\\
		&= \bigangle{s_{\kappa^+},\pi_d(s_\tau s_L)-\pi_d(s_D)}.
	\end{align*}
	Using \eqref{eq:s_D}, we conclude that
	\begin{align*}
		[s_\alpha] \Phi_{t,j}^\beta = \bigangle{s_{\kappa^+},\pi_d(s_\tau s_L)-\pi_d(s_\rho s_T)},
	\end{align*}
	so that
	\begin{align*}
		[s_\alpha] \Phi_{t,j}^\beta \ge 0,
	\end{align*}
	which comes from the Schur positivity of $\pi_d(s_\tau s_L)-\pi_d(s_\rho s_T)$ derived in \eqref{eq:tau*L-rho*T}.
	
	\textit{Case 2: $\beta\subseteq P_{t'}^{(j)}$ but $\beta\not\subseteq R_{t'}$.} Here $\len(\beta)=t'+1$ and $1\le \beta_{t'+1}\le j$. Let $\kappa':=(d-\beta_{t'+1},\ldots,d-\beta_1)$ be the rotated completement of $\beta$ in $R_{t'+1}$, so that $\kappa'_1\le d-1$. In a similar way as we get \eqref{eq:rotation}, rotation gives
	\begin{align*}
		s_{R_{t'+1}/\beta} = s_{\kappa'},\qquad\qquad s_{P_{t'}^{(j)}/\beta} = s_{\kappa'/(d')}.
	\end{align*}
	Now,
	\begin{align*}
		A_{t',p}(\alpha) = \bigangle{s_\alpha,s_{P_{t'}^{(j)}/\beta}s_{R_p}}= \bigangle{s_{\alpha/R_p},s_{P_{t'}^{(j)}/\beta}}= \bigangle{s_\tau s_\rho,s_{\kappa'/(d')}}= \bigangle{s_{\kappa'},h_{d'}s_\tau s_\rho}.
	\end{align*}
	Next, since the skew diagram $R_{t'}/\beta$ is invalid,
	\begin{align*}
		B_{t',p}(\alpha) = \bigangle{s_\alpha,s_{R_{t'}/\beta}s_{P_{p}^{(j)}}}=0.
	\end{align*}
	Finally,
	\begin{align*}
		B_{t'+1,p-1}(\alpha) = \bigangle{s_\alpha,s_{R_{t'+1}/\beta}s_{P_{p-1}^{(j)}}} = \bigangle{s_{\alpha/P_{p-1}^{(j)}},s_{R_{t'+1}/\beta}} = \bigangle{s_{\kappa'},s_D}.
	\end{align*}
	By \eqref{eq:alpha-no-d-Phi-simple},
	\begin{align*}
		[s_\alpha] \Phi_{t,j}^\beta = \bigangle{s_{\kappa'},h_{d'}s_\tau s_\rho-s_D} = \bigangle{s_{\kappa'},\pi_{d-1}(h_{d'}s_\tau s_\rho)-\pi_{d-1}(s_D)},
	\end{align*}
	with the second equality coming from the fact that $\kappa'_1\le d-1$. Using \eqref{eq:s_D}, we then get
	\begin{align*}
		[s_\alpha] \Phi_{t,j}^\beta = \bigangle{s_{\kappa'},\pi_{d-1}(h_{d'}s_\tau s_\rho)-\pi_{d-1}(s_\rho s_T)} = \bigangle{s_{\kappa'},s_\rho(h_{d'}s_\tau - s_T)}.
	\end{align*}
	It follows that
	\begin{align*}
		[s_\alpha] \Phi_{t,j}^\beta \ge 0,
	\end{align*}
	because $h_{d'}s_\tau - s_T$ is Schur positive according to \eqref{eq:s_T}, and therefore so is $s_\rho(h_{d'}s_\tau - s_T)$.
	
	\textit{Case 3: $\beta\subseteq R_{t'+1}$ but $\beta\not\subseteq P_{t'}^{(j)}$.} Here $\len(\beta)=t'+1$ and $j+1\le \beta_{t'+1}\le d$. Since the skew diagrams $P_{t'}^{(j)}/\beta$ and $R_{t'}/\beta$ are invalid,
	\begin{align*}
		A_{t',p}(\alpha) = B_{t',p}(\alpha) = 0.
	\end{align*}
	Let $\kappa'$ be as in the second case so that $\kappa'_1=d-\beta_{t'+1}$. We then use \eqref{eq:alpha-no-d-Phi-simple} to get
	\begin{align*}
		[s_\alpha] \Phi_{t,j}^\beta = -B_{t'+1,p-1}(\alpha) = -\bigangle{s_{\kappa'},s_{\alpha/P_{p-1}^{(j)}}} = -c_{\kappa',P_{p-1}^{(j)}}^{\alpha},
	\end{align*}
	where we have recalled \eqref{eq:Schur-c-2} for the last equality. If $p=0$, then the above is $0$ by default. Let $p\ge 1$ and assume that the above does not vanish. Then according to \eqref{eq:LR-facts-ineq},
	\begin{align*}
		\alpha_p \le \kappa'_1+(P_{p-1}^{(j)})_p = \kappa'_1 + j = d-\beta_{t'+1}+j < d.
	\end{align*}
	However, by definition we have $\alpha_p = d+\tau_p > d$, resulting in a contradiction. Thus, $c_{\kappa',P_{p-1}^{(j)}}^{\alpha}=0$, so that
	\begin{align*}
		[s_\alpha] \Phi_{t,j}^\beta = 0.
	\end{align*}
	
	\textit{Case 4: $\beta\not\subseteq R_{t'+1}$.} It is clear that
	\begin{align*}
		A_{t',p}(\alpha) = B_{t',p}(\alpha) = B_{t'+1,p-1}(\alpha) = 0,
	\end{align*}
	because none of the skew diagrams $P_{t'}^{(j)}/\beta$, $R_{t'}/\beta$, $R_{t'+1}/\beta$ is valid. We conclude from \eqref{eq:alpha-no-d-Phi-simple} that
	\begin{align*}
		[s_\alpha] \Phi_{t,j}^\beta = 0
	\end{align*}
	in this final case.
\end{proof}

\section{Strong $q$-log-convexity of $d$-Hoggatt polynomials}\label{sec:d-Hoggatt}

For any alphabet $X$, define a formal power series in $\Lambda[[q]]$:
\begin{align}
	\dcH_X(q) = \dcH_X^{(d)}(q) := \sum_{k\ge 0} s_{R_k}(X)q^k.
\end{align}

\begin{lemma}\label{le:dcH-1}
	We have
	\begin{align}
		\dcH_{1_N}^{(d)}(q) = H_N^{(d)}(q).
	\end{align}
\end{lemma}

\begin{proof}
	Note that $s_{R_k}(1_N) = 0$ for all $k>N$ because in these cases $R_k$ has more than $N$ rows and hence requires more than $N$ variables to ensure nonvanishing. Therefore,
	\begin{align*}
		\dcH_{1_N}(q) = \sum_{k= 0}^N s_{R_k}(1_N)q^k = \sum_{k= 0}^N \qangle{N}{k}q^k = H_N(q),
	\end{align*}
	where we have used \eqref{eq:Hoggatt-Schur-1}. 
\end{proof}

Let $X_N=(x_1,\ldots,x_N)$ and $Y_M=(y_1,\ldots,y_M)$ be two alphabets. Denote
\begin{align*}
	X_N\oplus Y_M := (x_1,\ldots,x_N,y_1,\ldots,y_M).
\end{align*}
By separating, in a semistandard tableau of shape $\lambda$, the cells whose entries refer to variables in $X_N$ and from those whose entries refer to variables in $Y_M$, we have
\begin{align}\label{eq:alphabet-splitting}
	s_{\lambda}(X_N\oplus Y_M) = \sum_{\mu \subseteq \lambda} s_{\mu}(X_N) s_{\lambda/\mu}(Y_M).
\end{align}

In the sequel, we treat a single variable $z$ as an alphabet $(z)$.

\begin{lemma}\label{le:R-alphabet-splitting}
	For $k\ge 0$,
	\begin{align}
		s_{R_k}(X_N \oplus z) = \sum_{j=0}^d s_{P_{k-1}^{(j)}}(X_N) z^{d-j}.
	\end{align}
\end{lemma}

\begin{proof}
	Note that
	\begin{align*}
		s_{\lambda/\mu}(z)=\begin{cases}
			z^{|\lambda|-|\mu|}, & \text{if $\lambda/\mu$ is a horizontal strip},\\
			0, & \text{otherwise}.
		\end{cases}
	\end{align*}
	By \eqref{eq:alphabet-splitting}, we have
	\begin{align*}
		s_{R_k}(X_N \oplus z) = \sum_{\mu\colon R_k/\mu\HS} s_\mu(X_N) z^{dk-|\mu|}.
	\end{align*}
	Since $R_k/\mu$ is a horizontal strip if and only if $\mu = P_{k-1}^{(j)}$ for a certain $j$ with $0\le j\le d$. The desired identity follows.
\end{proof}

Our next result is related to the auxiliary functions $\Phi$ defined in \eqref{eq:Phi-def}.

\begin{lemma}
	Let $A$ and $B$ be disjoint alphabets and put $X=A\oplus B$. Then
	\begin{align}\label{eq:Phi-AB}
		\dcH_{X\oplus z}(q) \dcH_B(q) - \dcH_X(q) \dcH_{B\oplus z}(q) = \sum_{r\ge 1} q^r \sum_{j=1}^{d-1} z^{d-j} \sum_\beta s_\beta(A) \Phi_{r-1,j}^{\beta}(B).
	\end{align}
\end{lemma}

\begin{proof}
	The coefficient of $q^0$ in $\dcH_{X\oplus z}(q) \dcH_B(q) - \dcH_X(q) \dcH_{B\oplus z}(q)$ is $1-1 = 0$. For $r\ge 1$, the coefficient of $q^r$ is
	\begin{align}
		c_r:=\sum_{k+l=r} \big(s_{R_k}(X\oplus z) s_{R_l}(B) - s_{R_k}(X) s_{R_l}(B\oplus z)\big).
	\end{align}
	Applying Lemma~\ref{le:R-alphabet-splitting}, we have
	\begin{align*}
		c_r = \sum_{j=0}^d z^{d-j}\sum_{k+l=r} \big(s_{P_{k-1}^{(j)}}(X)s_{R_l}(B)-s_{R_k}(X)s_{P_{l-1}^{(j)}}(B)\big).
	\end{align*}
	When $j=0$, we have
	\begin{align*}
		&\sum_{k+l=r} \big(s_{P_{k-1}^{(0)}}(X)s_{R_l}(B)-s_{R_k}(X)s_{P_{l-1}^{(0)}}(B)\big)\\
		&\qquad = \sum_{k+l=r} \big(s_{R_{k-1}}(X)s_{R_l}(B)-s_{R_k}(X)s_{R_{l-1}}(B)\big)\\
		&\qquad = \sum_{k+l=r-1} \big(s_{R_k}(X)s_{R_l}(B)-s_{R_k}(X)s_{R_l}(B)\big) = 0.
	\end{align*}
	When $j=d$, we have
	\begin{align*}
		&\sum_{k+l=r} \big(s_{P_{k-1}^{(d)}}(X)s_{R_l}(B)-s_{R_k}(X)s_{P_{l-1}^{(d)}}(B)\big)\\
		&\qquad = \sum_{k+l=r} \big(s_{R_k}(X)s_{R_l}(B)-s_{R_k}(X)s_{R_l}(B)\big) = 0.
	\end{align*}
	Thus,
	\begin{align*}
		c_r = \sum_{j=1}^{d-1} z^{d-j}\sum_{k+l=r-1} \big(s_{P_{k}^{(j)}}(X)s_{R_l}(B)-s_{R_k}(X)s_{P_{l}^{(j)}}(B)\big).
	\end{align*}
	Recalling that $X=A\oplus B$, we apply \eqref{eq:alphabet-splitting} and obtain
	\begin{align*}
		s_{P_{k}^{(j)}}(X) = s_{P_{k}^{(j)}}(A\oplus B) = \sum_\beta s_\beta(A) s_{P_{k}^{(j)}/\beta}(B),
	\end{align*}
	and
	\begin{align*}
		s_{R_k}(X) = s_{R_k}(A\oplus B) = \sum_\beta s_\beta(A) s_{R_k/\beta}(B).
	\end{align*}
	It follows that
	\begin{align*}
		c_r &= \sum_{j=1}^{d-1} z^{d-j} \sum_\beta s_\beta(A) \sum_{k+l=r-1} \big(s_{P_{k}^{(j)}/\beta}(B)s_{R_l}(B)-s_{R_k/\beta}(B)s_{P_{l}^{(j)}}(B)\big)\\
		&= \sum_{j=1}^{d-1} z^{d-j} \sum_\beta s_\beta(A) \Phi_{r-1,j}^{\beta}(B),
	\end{align*}
	as desired.
\end{proof}

Now we are in a position to finalize the proof of Theorem~\ref{thm:main2}.

\begin{proof}[Proof of Theorem~\ref{thm:main2}]
	Let $n\ge m\ge 1$. We set $A=1_{n-m+1}$ and $B=1_{m-1}$ so that $X=1_{n}$. Also, set $z=1$. Then Lemma~\ref{le:dcH-1} tells us that
	\begin{align*}
		H_{n+1}(q)H_{m-1}(q) - H_{n}(q)H_{m}(q) = \dcH_{X\oplus z}(q) \dcH_B(q) - \dcH_X(q) \dcH_{B\oplus z}(q).
	\end{align*}
	Since $\Phi_{r-1,j}^{\beta}$ is Schur positive for every $j$ with $1\le j\le d-1$ according to Theorem~\ref{th:Phi-Schur-positivity}, we conclude from \eqref{eq:Phi-AB} that the above is in $\mathbb{N}[q]$, thereby implying the strong $q$-log-convexity of $(H_n(q))_{n\geq0}$.
\end{proof}

\subsection*{Acknowledgements}

Shane Chern was supported by the Austrian Science Fund (No.~10.55776/F1002). Wenle Shi was supported by the China Scholarship Council (No.~202506060077).

\bibliographystyle{amsplain}

\end{document}